\documentclass[12pt]{amsart}
\usepackage{indentfirst, latexsym, bm, amsmath, amsthm, amssymb}
\usepackage[colorlinks, linkcolor=blue,citecolor=red]{hyperref}
\usepackage{upgreek}
\usepackage{url}

\theoremstyle{definition}
\theoremstyle{remark}
\numberwithin{equation}{section}
\newtheorem{theorem}{Theorem}[section]
\newtheorem{lemma}[theorem]{Lemma}
\newtheorem{corollary}[theorem]{Corollary}
\newtheorem{proposition}[theorem]{Proposition}

\newtheorem{remark}[theorem]{Remark}
\newtheorem{example}[theorem]{Example}
\newcommand{\on}{\overline{\nabla}}
\newcommand{\D}{\mathcal{D}}
\newcommand{\XH}{X^{\mathcal H}}
\newcommand{\R}{\mathbb{R}}
\newcommand{\C}{\mathbb{C}}
\newcommand{\la}{\langle}
\newcommand{\ra}{\rangle}
\newcommand{\p}{\partial}

\begin{document}
\title[Contact blow-up analysis of the Legendrian mean curvature flow]{Contact Blow-up Analysis of the
Legendrian Mean Curvature Flow and a Smale-type Theorem for Strict Contactomorphisms of Sasakian
Spheres}

\author{Shu-Cheng Chang}
\address{Shanghai Institute for Mathematics and Interdisciplinary Sciences, Shanghai 200433, China}
\email{scchang@math.ntu.edu.tw}
\author{Chin-Tung Wu}
\address{Department of Applied Mathematics, National Pingtung University, Pingtung 90003, Taiwan}
\email{ctwu@mail.nptu.edu.tw}
\author{Liuyang Zhang}
\address{School of Mathematical Science, Chongqing University of Technology, Chongqing 400054, China}
\email{zhangliuyang@cqut.edu.cn}


\begin{abstract}
We study the Legendrian mean curvature flow in $\eta$-Einstein Sasakian manifolds by parabolic
rescalings which preserve the contact structure. In a compact $\eta$-Einstein Sasakian manifold,
regular or not, the contact tangent flows at a Type I singularity are, up to vertical translations of their
components, self-similar flows of Legendrian self-shrinkers in the Heisenberg group, and they are unions of
Legendrian planes. When the Reeb orbits are not closed, as on the
irregular Sasaki--Einstein manifolds, the leaf space of the Reeb foliation is not a K\"ahler manifold or orbifold,
and this result does not follow by projecting the flow to a Lagrangian mean curvature flow; its proof combines
local transverse estimates with a Gaussian density localised in a flow box of the Reeb foliation. In a regular manifold a Type I singularity can only form at a point of transverse
Gaussian density at least $2$. We apply the analysis to strict
contactomorphisms $\varphi$ of the standard Sasakian sphere $\mathbb S^{2k+1}$, whose Legendrian
graphs lie in the $\eta$-Einstein manifold $(\mathbb S^{2k+1}\times\mathbb S^{2k+1})/\mathbb S^1$. The
Jacobian of a Legendrian graph obeys a reaction--diffusion equation whose quadratic term is a sum
of squares when $k=1$ and is indefinite when $k\geq2$. For every strict contactomorphism of
$\mathbb S^3$, and for every pinched strict contactomorphism of $\mathbb S^{2k+1}$, $k\geq2$, the
contact blow-ups of the flow of its Legendrian graph are flat, the flow exists for all time and,
up to a rotation along the Reeb flow, converges exponentially to the graph of an element of
$U(k+1)$. In particular such a $\varphi$ is isotopic to $U(k+1)$ through strict
contactomorphisms.
\end{abstract}
\maketitle

\section{Introduction}

\subsection*{The Smale conjecture and mean curvature flows of graphs}
Smale proved that the orthogonal group $O(3)$ is a deformation retract of the diffeomorphism group
of $\mathbb S^2$ \cite{sm1}. That $\mathrm{Diff}(\mathbb S^3)$ has the homotopy type of $O(4)$ is the Smale
conjecture, proved by Hatcher \cite{hat}. In higher dimensions the corresponding statement fails in
general: for instance, by Cerf's pseudo-isotopy theorem, $\pi_0\mathrm{Diff}^+(\mathbb S^6)$ is the group $\Gamma_7$ of twisted
$7$-spheres \cite[Introduction, Corollaire~2]{cerf}, which by Smale's $h$-cobordism theorem is the group $\Theta_7$ of
homotopy $7$-spheres, of order $28$ \cite[Section~1]{km}, while $SO(7)$ is connected, and Watanabe \cite{wat} found rational homotopy
of $\mathrm{Diff}(\mathbb S^4)$ which does not come from $O(5)$. For the symplectic counterparts,
Moser's argument \cite{mos} shows that $\mathrm{Diff}^+(\mathbb S^2)$ retracts onto the group of
area-preserving diffeomorphisms, so that Smale's theorem is a statement about symplectomorphisms of
$\mathbb{CP}^1$; Gromov proved that the symplectomorphism group of $\mathbb{CP}^2$ retracts onto
$PU(3)$ \cite{gro}, while for $\mathbb{CP}^k$, $k\geq3$, the question is open to our knowledge (see
\cite[Section~1]{mw}).

M.-T. Wang proposed to deform a symplectomorphism $f$ by the mean curvature flow of its graph,
a Lagrangian submanifold of the product. For area-preserving diffeomorphisms of $\mathbb S^2$ the
flow exists for all time and converges to the graph of an isometry \cite[Theorem~A and
Theorem~3.1]{w1}, which gives a new proof of Smale's theorem \cite[Corollary~C]{w1}. Medo\v{s} and
Wang extended this to $\mathbb{CP}^k$: a $\Lambda$-pinched symplectomorphism with $\Lambda<\Lambda(k)$ is
symplectically isotopic to a biholomorphic isometry along the flow \cite[Theorem~1 and
Corollary~1]{mw}. The pinching enters through a quadratic form in the second fundamental form which
has to be non-negative \cite[Proposition~3]{mw}; for $k=1$ it is not needed \cite[Remark~2]{mw}.

\subsection*{Strict contactomorphisms and their Legendrian graphs}
The contact counterparts of the complex projective spaces are the Sasakian spheres
$(\mathbb S^{2k+1},\eta_0)$, whose Reeb flow is the Hopf action $u\mapsto e^{\sqrt{-1}t}u$ and whose
leaf space is $\mathbb{CP}^k$. A \emph{strict contactomorphism} is a diffeomorphism $\varphi$ with
$\varphi^*\eta_0=\eta_0$. It commutes with the Hopf action and covers a symplectomorphism $\tilde\varphi$ of
$(\mathbb{CP}^k,\omega_{\mathrm{FS}})$, and two strict contactomorphisms covering the same $\tilde\varphi$
differ by a rotation along the Reeb flow. A deformation of $\varphi$ through strict contactomorphisms is
therefore a deformation of $\tilde\varphi$ together with a deformation of the phase of $\varphi$ along
the Reeb circles, and the phase is invisible on $\mathbb{CP}^k$. (For the $3$-sphere, the group of coorientation-preserving
contactomorphisms of the standard contact structure is homotopy equivalent to $U(2)$. This follows from Hatcher's
theorem \cite{hat} and the contractibility of the space of tight contact structures on $\mathbb S^3$ fixed at a point,
which was stated in \cite{el} and proved in \cite{em21}; see also \cite[Theorem~4]{cs16}. Complex conjugation
preserves the contact structure but reverses its coorientation.)

The graph of $\varphi$ is invariant under the diagonal circle action on
$\mathbb S^{2k+1}\times\mathbb S^{2k+1}$, and its quotient $\Sigma_\varphi$ is a Legendrian submanifold of
\[
M_k=(\mathbb S^{2k+1}\times\mathbb S^{2k+1})/\mathbb S^1 ,
\]
with the contact form induced by $\eta_0\oplus(-\eta_0)$. With its natural metric, $M_k$ is a regular
$\eta$-Einstein Sasakian manifold over $\mathbb{CP}^k\times\mathbb{CP}^k$ which becomes Sasaki--Einstein
after a $\mathcal D$-homothety (Lemma~\ref{lem-M}); $M_1$ is diffeomorphic to $\mathbb S^2\times\mathbb S^3$,
while for $k\geq2$ the manifold $M_k$ is not diffeomorphic to $\mathbb{CP}^k\times\mathbb S^{2k+1}$
(Remark~\ref{rem-topology}). The Legendrian graphs in $M_k$ are exactly the submanifolds
$\Sigma_\varphi$ (Proposition~\ref{pro-graph}). We deform them by the \emph{Legendrian mean curvature
flow}
\begin{equation}\label{c}
\frac{\p F}{\p t}=H-2\theta\,\xi
\end{equation}
of Smoczyk \cite[Definition~3.5]{s1}, where $\xi$ is the Reeb field and $\theta$ the Legendrian angle
\cite[Definition~3.4]{s1}, $H=\Phi\nabla\theta$. The mean curvature form of a Legendrian submanifold of an
$\eta$-Einstein Sasakian manifold is closed \cite[Lemma~2.8]{s1}, and \eqref{c} presupposes that it is exact;
Remark~\ref{rem-nonexact} describes Smoczyk's flow in the general case. The mean curvature flow does not preserve the Legendrian condition, and \eqref{c}
is the flow which does, in $\eta$-Einstein Sasakian manifolds; with the normalisation
$d\eta=2g(\cdot,\Phi\cdot)$ of \cite{cwzz} and \cite[(2.5)]{cqz} the factor $2$ is forced \cite[Lemma~3.1]{s1}. Long-time
existence and convergence of \eqref{c} under stability conditions inspired by the Thomas--Yau
conjecture \cite{ty} were proved in \cite[Theorems~1.2 and~1.3]{chwu}. Along \eqref{c},
$\p_t\theta=\Delta\theta+(K+2)\theta$ \cite[Corollary~3.3]{s1}, \cite[(3.2)]{chwu}, where $K+2$ is the transverse
Einstein constant; the phase of $\varphi$ is carried by $\theta$, and the normalised flow
\eqref{c-norm}, in which $\theta$ is replaced by $\theta-\overline\theta$, removes the rotation along the Reeb
flow caused by the mean value $\overline\theta$ of $\theta$.

\subsection*{The role of the Legendrian setting}
The present paper is an application of the analysis of \eqref{c} carried out in
\cite{cwzz,cqz,cqz26}, and it provides the link between that analysis and the flow. Three features of
the Legendrian setting are used throughout. First, the rescalings compatible with a Sasakian structure
are the $\mathcal D$-homothetic deformations of Tanno \cite[(1.1)]{tan}, which have no K\"ahler counterpart;
they change a $\Phi$-sectional curvature $-3+s$ into $-3+s/\kappa$ \cite[Lemma~6.1]{tan}, so that the blow-up model of a
Sasakian manifold is not a Euclidean space but the Heisenberg group, the Sasakian space form of
$\Phi$-sectional curvature $-3$, on which the rescalings act by the contact dilations
$\delta_\lambda(x,y,z)=(\lambda x,\lambda y,\lambda^2z)$. The self-similar solutions with respect to $\delta_\lambda$ are
the Legendrian self-shrinkers of \cite{cwzz,cqz}, and those with respect to the left translations are the
translating solitons of \cite{cqz26}; the rigidity of the former \cite[Theorem~A]{cqz} is used below.
Second, a Legendrian deformation is generated by a single function, with no flux obstruction, and
\eqref{c} is the deformation generated by $-2\theta$; accordingly the Legendrian angle, rather than a
quantity built from the second fundamental form, is the natural object to control, and on a
Legendrian self-shrinker it is a coordinate, $4\theta=z$ \cite[Proposition~3.3(i)]{cqz}. Third, the objects are
Legendrian: \eqref{c} moves a closed Legendrian submanifold through Legendrian immersions, and closed
Legendrian submanifolds carry a rich isotopy theory \cite{ees,gei}; in the present application it moves
a strict contactomorphism through strict contactomorphisms and determines its phase.

We state precisely what is new. The projection of \eqref{c} to $\mathbb{CP}^k\times\mathbb{CP}^k$ is the mean curvature
flow of the graph of $\tilde\varphi$ (Lemma~\ref{lem-proj}), so the isotopy of $\tilde\varphi$ which our flow induces is
the one of \cite{w1,mw}. Since $H^1(\mathbb{CP}^k;\R)=0$, a symplectic isotopy of $\mathbb{CP}^k$ is Hamiltonian and lifts
to a strict contact isotopy, and two strict contactomorphisms covering the same map differ by a constant
rotation along the Reeb flow; hence the isotopy of $\varphi$ which our flow produces is the lift of the isotopy of
\cite{w1,mw} composed with a time-dependent rotation along the Reeb flow (Remark~\ref{rem-lift}), and an isotopy of
$\varphi_0$ to $U(k+1)$ through strict contactomorphisms can also be obtained in this way. The statements of this paper
that concern the Legendrian flow itself, rather than its projection, are Theorem~\ref{thm-typeI} and
Corollary~\ref{cor-planes-intro}, which describe the contact tangent flows of arbitrary closed Legendrian submanifolds
with exact mean curvature form in $\eta$-Einstein Sasakian manifolds, regular or not, including their vertical organisation
(the relation $4\theta=z$ on the limit and its centring about $(p,T)$), and the description of the contact blow-ups of
Legendrian graphs in Theorem~\ref{thm-flat}. When the Reeb flow is not periodic, Theorem~\ref{thm-typeI} and
Corollary~\ref{cor-planes-intro} do not follow by projecting the flow to a leaf space; see the discussion after
Theorem~\ref{thm-typeI}. Corollary~\ref{cor-typeI-intro} is proved through the projected flow of a regular manifold; we
include it because it delimits where Type I singularities can form. Theorem~\ref{thm-smale} and
Corollary~\ref{cor-smale} recover \cite{w1,mw} in the contact setting; we include them because they are the natural
situation in which the contact rescalings can be analysed without any assumption on the singularities. In the regular
case the transverse monotonicity formula of Section~\ref{sec-blowup} is the monotonicity formula of the projected flow
written in contact terms; what is specific to the contact setting is its use along contact rescalings, which retain the
vertical structure that the projection forgets. We use the projection as a technical device in the linear algebra of
the quadratic form in Proposition~\ref{prop:jacobian}, in the definition of the transverse Gaussian density through an
embedding of the leaf space, in the interior estimates, and in White's local regularity theorem \cite{wh}.

\subsection*{Contact blow-ups and Type I singularities}
The rescaled flows are controlled by a \emph{transverse} Gaussian density built from the horizontal
potential $f=\frac14|\XH|^2$, the function $f$ of \cite[(4.3)]{cqz}, where $\XH$ is the horizontal part of the generator $X$ of
$\delta_\lambda$. In the Heisenberg group $\R^{2n+1}$ the kernel $(4\pi(T-t))^{-n/2}e^{-f/(T-t)}$ satisfies
the exact monotonicity formula
\[
\frac{d}{dt}\int_{\Sigma_t}\frac{e^{-f/(T-t)}}{(4\pi(T-t))^{n/2}}\,d\mu
=-\int_{\Sigma_t}\Big|H+\frac{(\overline\nabla f)^\perp}{T-t}\Big|^2\frac{e^{-f/(T-t)}}{(4\pi(T-t))^{n/2}}\,d\mu ,
\qquad \overline\nabla f=\frac12\XH,
\]
along \eqref{c} (Proposition~\ref{prop:mono-flat}); a horizontal potential centred at each point of a
regular Sasakian manifold gives an almost monotone version (Proposition~\ref{prop:mono}), and in an arbitrary
Sasakian manifold a version localised in a flow box of the Reeb foliation holds up to an integrable error
(Proposition~\ref{prop:mono-local}). The kernel
is constant along the Reeb orbits, so the vertical velocity $-2\theta\xi$ does not enter, and the full
generator $X$ cannot be used (Remark~\ref{rem-horizontal-kernel}). A Legendrian immersion into
$\R^{2n+1}$ is a \emph{Legendrian self-shrinker} \cite{cwzz}, \cite[Definition~3.2]{cqz} if
\begin{equation}\label{shrinker}
H-2\theta\,\xi=-\frac12X^\perp ;
\end{equation}
equivalently, $F(\cdot,s)=\delta_{\sqrt{-s}}F(\cdot,-1)$, $s<0$, solves \eqref{c} up to time-dependent tangential
reparametrisation.

\begin{theorem}[Type I contact tangent flows]\label{thm-typeI}
Let $M$ be a compact $\eta$-Einstein Sasakian manifold, not necessarily regular, and let $F:\Sigma\times[0,T)\to
M$ be a solution of \eqref{c} with $\Sigma$ closed and exact mean curvature form, the Legendrian angle $\theta$
being normalised by \eqref{d}, which develops a Type I singularity at time $T$ at a point $p\in M$. Then there is
a sequence $\lambda_i\to\infty$ such that the contact parabolic rescalings of scale $\lambda_i$ at $(p,T)$ converge,
smoothly on compact subsets of $\R^{2n+1}\times(-\infty,0)$ and as immersions with multiplicity, to a non-empty
flow of proper Legendrian immersions in the Heisenberg group, with finitely many connected components, which
satisfies
\[
H-2\theta\,\xi=\frac{1}{2s}X^{\perp},\qquad -\infty<s<0,
\]
where $\theta$ is the limit of the Legendrian angle; that is, up to time-dependent tangential reparametrisation,
$F_\infty(\cdot,s)=\delta_{\sqrt{-s}}F_\infty(\cdot,-1)$, and $F_\infty(\cdot,-1)$ is a Legendrian
self-shrinker. Along an arbitrary sequence $\lambda_i\to\infty$, a subsequence converges to a flow of the same kind
each of whose connected components is a vertical translate of such a self-similar flow.
\end{theorem}

Theorem~\ref{thm-typeI} does not assume that the Reeb flow is periodic. A Sasakian manifold is quasi-regular if the
Reeb orbits are closed and irregular otherwise; in the irregular case the orbits are dense in the orbits of a torus,
the closure of the Reeb flow in the isometry group of $M$, which is a compact connected abelian Lie group \cite{bg}. If the Reeb orbits are closed but the Reeb action is not free, the leaf space is a K\"ahler orbifold;
the proof given in Section~\ref{sec-typeI} uses a smooth leaf space, and such manifolds are covered by
Section~\ref{sec-general} as well. Irregular $\eta$-Einstein manifolds exist in abundance: the Sasaki--Einstein manifolds $Y^{p,q}$ of
\cite{gmsw04} on $\mathbb S^2\times\mathbb S^3$, with $p,q$ coprime and $0<q<p$, are irregular whenever $4p^2-3q^2$ is not a square \cite{ms06}. For such $M$ the leaf
space is not a K\"ahler manifold or orbifold, the Legendrian flow has no global projection to a Lagrangian mean
curvature flow, and the transverse density $\Theta$ is not defined. In its place we use a density localised in a flow box of the Reeb foliation, whose error from the cut-off in
the Reeb direction is integrable in time (Proposition~\ref{prop:mono-local}); the derivative estimates are obtained
by the maximum principle on the whole closed flow in the $\mathcal D$-homothetic deformations of $M$, whose
geometry is controlled uniformly through the local leaf spaces (Lemma~\ref{lem-compact-general}). To our knowledge
Theorem~\ref{thm-typeI} is the first blow-up result for the Legendrian mean curvature flow in irregular Sasakian
manifolds; in this case it does not follow from a Lagrangian statement by projection.

\begin{corollary}\label{cor-planes-intro}
In Theorem~\ref{thm-typeI}, every tangent flow is a union of finitely many Legendrian $n$-planes, counted with
multiplicity.
\end{corollary}

Corollary~\ref{cor-planes-intro} follows from Theorem~\ref{thm-typeI} and \cite[Theorem~A]{cqz}, applied to each component,
since the Legendrian angle stays bounded on closed solutions of \eqref{c}. In particular, the non-planar
Legendrian self-shrinkers of \cite{cwzz,cqz} do not arise as Type I tangent flows of closed
solutions. Combined with White's local regularity theorem, this restricts where a Type I singularity
can form.

\begin{corollary}[Density at Type I singular points]\label{cor-typeI-intro}
In the situation of Theorem~\ref{thm-typeI}, suppose in addition that $M$ is regular. Then the transverse Gaussian
density $\Theta(p,T)$ of the flow at every singular point $p$ is an integer, and $\Theta(p,T)\geq2$. In particular, no Type I singularity forms at a
point of transverse Gaussian density less than $2$.
\end{corollary}

The transverse density counts all sheets of $\Sigma_t$ which approach the Reeb orbit through $p$, so a
Type I singular point is a point where at least two sheets come together over one Reeb orbit.
Corollary~\ref{cor-typeI-intro} is proved through the projected flow, and we do not know whether Type I
singularities can be excluded altogether. Two situations should be distinguished. If $K+2\neq0$ and the
projected Lagrangian immersions are embedded, they are monotone, because their mean curvature form is exact
\cite[Corollary~2.7]{ev22}, and Evans proved that monotone Lagrangian mean curvature flows in K\"ahler--Einstein
manifolds with non-zero Einstein constant develop no Type I singularities \cite[Theorems~1.2 and~3.2]{ev22}.
Monotonicity is defined there for embedded Lagrangians \cite[Section~2.1]{ev22}, whereas the projection of a
Legendrian immersion has self-intersections as soon as the immersion has a double point or a Reeb chord, in
particular for many embedded Legendrian submanifolds. The present argument does not establish full Type I exclusion
for immersed projections.
An approach through densities would require separating the sheets over one Reeb orbit, that is, a local
regularity theorem in terms of a density localised in the direction of the Reeb field, since the transverse
kernel is blind to that direction (Remark~\ref{rem-sheets}). In
the Lagrangian setting the tangent flows at Type I singularities of zero-Maslov class flows are unions of planes as
well \cite[Theorem~A]{nev07}, \cite[Corollary~3.5]{nev11}.

\subsection*{The Jacobian of a Legendrian graph and the role of the dimension}
For Legendrian graphs in $M_k$ no assumption on the type of the singularity is needed. The key is the
\emph{Jacobian} $\Omega$, the value on the oriented Legendrian tangent plane of the transverse volume form
of the first factor of $\mathbb{CP}^k\times\mathbb{CP}^k$, normalised so that $0<\Omega\leq1$ on graphs, with
$\Omega=1$ exactly on graphs of isometries (Section~\ref{sec-jac}). It depends only on the tangent plane and the
transverse K\"ahler structure, so it is invariant under the contact rescalings (Lemma~\ref{lem-Omega-inv}),
and the vertical velocity $-2\theta\xi$ does not enter its evolution (Remark~\ref{rem-angle}). At a point
where $d\tilde\varphi$ has the singular values $\lambda_{2l-1}^{\pm1}$,
\begin{equation}\label{jac-intro}
\Big(\frac{\p}{\p t}-\Delta\Big)\Omega=Q(\lambda,h)\,\Omega+4\,\Omega\sum_{l=1}^k\big(1-s_l^{-2}\big),\qquad
s_l=\tfrac12\big(\lambda_{2l-1}+\lambda_{2l-1}^{-1}\big)\geq1,
\end{equation}
where $Q(\lambda,h)$ is the quadratic form of Medo\v{s} and Wang in the second fundamental form
(Proposition~\ref{prop:jacobian}). The dimension enters here. For $k=1$, $Q=2|B|^2-|H|^2$ is a sum of
squares and $Q\geq\frac23|B|^2$. For $k\geq2$ the form $Q$ is not signed (Example~\ref{ex-negative}); it is
non-negative, with a gain $\delta|B|^2$, when the singular values are pinched \cite[Proposition~3]{mw},
and the pinching is preserved along the flow (Proposition~\ref{prop:pinch}). In both cases
$(\p_t-\Delta)\Omega\geq\delta|B|^2\Omega$, and weighting the transverse density by $1-\Omega$ yields an almost
monotone quantity whose decrease controls $\int|B|^2$ at every scale (Proposition~\ref{prop:weighted}).

\begin{theorem}[Flat contact blow-ups of Legendrian graphs]\label{thm-flat}
Let $k\geq1$ and let $F:\Sigma\times[0,T)\to M_k$, $T<\infty$, be a solution of \eqref{c} starting from a
Legendrian graph; if $k\geq2$, assume that $\Omega\geq c_0(k)$ on $\Sigma_0$, with the constant
$c_0(k)<1$ of Section~\ref{sec-jac}. Then for every $p\in M_k$:
\begin{itemize}
\item[(i)] there are $\lambda_j\to\infty$ and $s_j\in[-2,-1]$ such that the contact parabolic rescalings of
scale $\lambda_j$ at $(p,T)$ at time $s_j$, composed with suitable rotations along the Reeb flow, converge
in $C^\alpha_{\rm loc}\cap W^{1,2}_{\rm loc}$ either to a Legendrian $2k$-plane of multiplicity one or to the
empty set;
\item[(ii)] the transverse Gaussian density of the flow at $(p,T)$ is at most $1$;
\item[(iii)] the second fundamental form stays bounded near $p$ as $t\to T$.
\end{itemize}
Consequently, such a flow exists for all time.
\end{theorem}

In (i), a Legendrian $2k$-plane is the horizontal lift of an affine Lagrangian $2k$-plane of $\R^{4k}$, that is, a left
translate of a Legendrian subspace $\{(w,0):w\in P\}$ of the Heisenberg group, $P$ a Lagrangian subspace of $\R^{4k}$. It
need not lie over a linear subspace, unlike the planes $\{(w,c):w\in P\}$ of Corollary~\ref{cor-planes-intro} and
\cite[Theorem~A]{cqz}. Theorem~\ref{thm-flat} does not use Theorem~\ref{thm-typeI}: compactness comes from the Jacobian instead of a Type I
bound.

\subsection*{A Smale-type theorem for Sasakian spheres}
A strict contactomorphism $\varphi$ of $(\mathbb S^{2k+1},\eta_0)$ is \emph{$\Lambda$-pinched} if the singular
values of $d\tilde\varphi$ with respect to the Fubini--Study metric lie in $[\Lambda^{-1},\Lambda]$, i.e.\ if $\tilde\varphi$ is
$\Lambda$-pinched in the sense of \cite[Definition~1]{mw}; equivalently, the singular values of $d\varphi$ on the
contact distribution with respect to the round metric lie in $[\Lambda^{-1},\Lambda]$.

\begin{theorem}[Deformation of strict contactomorphisms]\label{thm-smale}
Let $k\geq1$ and let $\varphi_0$ be a strict contactomorphism of $(\mathbb S^{2k+1},\eta_0)$; if $k\geq2$, assume
that $\varphi_0$ is $\Lambda$-pinched with $\Lambda\leq\Lambda_*(k)$, the constant of Section~\ref{sec-jac}. The normalised
Legendrian mean curvature flow in $M_k$ starting from $\Sigma_{\varphi_0}$ exists smoothly for all $t\geq0$;
each $\Sigma_t$ is the Legendrian graph $\Sigma_{\varphi_t}$ of a strict contactomorphism $\varphi_t$, and
$\Sigma_t$ converges smoothly and exponentially, as $t\to\infty$, to $\Sigma_{\varphi_\infty}$ with
$\varphi_\infty\in U(k+1)$. The solution of \eqref{c} from $\Sigma_{\varphi_0}$ differs from the normalised flow by
a rotation along the Reeb flow.
\end{theorem}

\begin{corollary}[Smale-type theorem]\label{cor-smale}
Every strict contactomorphism of the standard Sasakian $3$-sphere, and, for $k\geq2$, every
$\Lambda$-pinched strict contactomorphism of the standard Sasakian $(2k+1)$-sphere with
$\Lambda\leq\Lambda_*(k)$, is isotopic, through strict contactomorphisms, to an element of $U(k+1)$.
\end{corollary}

For $k=1$ the induced maps $\tilde\varphi_t$ give the isotopy of an area-preserving diffeomorphism of $\mathbb S^2$ to an
isometry of \cite[Corollary~C]{w1}; for $k\geq2$ the induced isotopy of $\tilde\varphi$ is the one of \cite[Theorem~1]{mw}.
Theorem~\ref{thm-smale} and Corollary~\ref{cor-smale} are statements about individual maps: we do not show that the isotopy
depends continuously on $\varphi_0$, which would be needed for a deformation retraction of the group of strict
contactomorphisms onto $U(k+1)$, the analogue of Smale's theorem itself. The proof of Theorem~\ref{thm-smale} has three steps: the graph condition, and for
$k\geq2$ the pinching, are preserved (Section~\ref{sec-jac}); no singularity forms, by
Theorem~\ref{thm-flat}; and the flow converges, by the estimates of Section~\ref{sec-conv}, which use the same
Jacobian.

\begin{remark}\label{rem-group}
For $k=1$ and $k=2$ the statement at the level of groups follows from classical results. Let
$\mathrm{Diff}(\mathbb S^{2k+1},\eta_0)$ be the group of strict contactomorphisms. Every symplectomorphism of
$(\mathbb{CP}^k,\omega_{\mathrm{FS}})$ lifts to it: it preserves $[\omega_{\mathrm{FS}}]$, hence the Chern class of the Hopf
bundle, and since $H^1(\mathbb{CP}^k;\R)=0$ the pulled-back connection is gauge equivalent to $\eta_0$. Two lifts differ
by a constant rotation along the Reeb flow. Thus $\mathrm{Diff}(\mathbb S^{2k+1},\eta_0)$ is a principal $\mathbb S^1$-bundle
over $\mathrm{Symp}(\mathbb{CP}^k,\omega_{\mathrm{FS}})$, a central extension of this group by the Reeb rotations, just as
$U(k+1)$ is a central extension of $PU(k+1)$ by $\mathbb S^1$. By the five lemma applied to the long exact sequences
of homotopy groups, the inclusion $U(k+1)\subset\mathrm{Diff}(\mathbb S^{2k+1},\eta_0)$ is a weak homotopy equivalence
whenever $PU(k+1)\subset\mathrm{Symp}(\mathbb{CP}^k,\omega_{\mathrm{FS}})$ is one. This is the case for $k=1$, by Smale's
theorem \cite{sm1} and Moser's argument \cite{mos}, and for $k=2$, by Gromov's theorem \cite{gro}; in both cases the
symplectomorphism group is connected, every symplectomorphism lifts and the fibre is connected, so that the whole
strict group, not only its identity component, is concerned. For the identity components this argument is
\cite[Section~4.3]{cs16}: the principal circle bundle over the Hamiltonian group is \cite[Proposition~14]{cs16}, the
case of $\mathbb S^5$ is \cite[Corollary~16]{cs16} and the case of $\mathbb S^3$ is the paragraph following it. For $k=1$,
DeTurck, Gluck, Lichtenfelz, Merling, Wang and Yang lift Wang's deformation retraction to a smooth deformation
retraction of the strict contactomorphism group of $\mathbb S^3$ onto $U(2)$ \cite[Theorem~A]{dglmwy}. These results are
background and not consequences of the Legendrian flow. For $k\geq3$ the corresponding question for $\mathbb{CP}^k$ is
open to our knowledge, and Corollary~\ref{cor-smale} gives the isotopy for pinched maps only.
\end{remark}

The paper is organised as follows. Section~\ref{sec-prelim} collects the conventions, the projection to
the leaf space, the Heisenberg model, the manifolds $M_k$ and their Legendrian graphs.
Section~\ref{sec-jac} derives the evolution of the Jacobian and the preservation of pinching.
Section~\ref{sec-blowup} contains the contact blow-up analysis: the rescalings and the transverse
monotonicity formula, the Type I tangent flows (Theorem~\ref{thm-typeI} and Corollary~\ref{cor-typeI-intro}),
the blow-ups of Legendrian graphs (Theorem~\ref{thm-flat}), and the removal of the regularity assumption from
Theorem~\ref{thm-typeI} (Section~\ref{sec-general}). Section~\ref{sec-conv} proves the convergence
and Theorem~\ref{thm-smale}.

\section{Preliminaries}\label{sec-prelim}

\subsection{Conventions}
We refer to \cite{bl,bg} for Sasakian geometry. A Sasakian manifold $(M^{2n+1},\eta,\xi,\Phi,g)$ has
contact form $\eta$, unit Reeb field
$\xi$, contact distribution $\D=\ker\eta$ and a $(1,1)$-tensor $\Phi$ with
\begin{equation}\label{sas}
\begin{gathered}
\Phi^2=-I+\eta\otimes\xi,\qquad d\eta(U,W)=2g(U,\Phi W),\\
\on_U\xi=-\Phi U,\qquad (\on_U\Phi)W=g(U,W)\xi-\eta(W)U .
\end{gathered}
\end{equation}
These conventions agree with those of \cite{bl}, see \cite[Theorem~6.3 and its proof]{bl}; there the exterior
derivative of a $1$-form carries a factor $\frac12$, so that the second identity reads $d\eta(U,W)=g(U,\Phi W)$.
It is $\eta$-Einstein if $\mathrm{Ric}=Kg+(2n-K)\eta\otimes\eta$ for a constant $K$; then the
transverse Ricci tensor is $\mathrm{Ric}^T=(K+2)g^T$, where $g=g^T+\eta\otimes\eta$. It is
Sasaki--Einstein if $K=2n$. We use the curvature convention
$\overline R(U,V)W=\on_U\on_VW-\on_V\on_UW-\on_{[U,V]}W$.

\begin{remark}[Notation]
In \cite{chwu} the Reeb field is denoted $\mathcal T$ and the contact distribution $\xi$; here they
are $\xi$ and $\D$. In \cite{s1,chwu} the tensor $\Phi$ has the opposite sign; with our sign, \eqref{sas}
holds with the exterior derivative $d\eta(U,W)=U\eta(W)-W\eta(U)-\eta([U,W])$. The Legendrian angle is
denoted $\theta$ (it is $\alpha$ in \cite{cwzz}); with these conventions $H=\Phi\nabla\theta$ and the flow
takes the form \eqref{c}. The conventions agree with those of \cite{cwzz,cqz}.
\end{remark}

\subsection{Legendrian immersions and the Legendrian mean curvature flow}\label{sec-flow}
Let $F:\Sigma^n\to M$ be Legendrian, $F^*\eta=0$. Then $N\Sigma=\Phi(T\Sigma)\oplus\R\xi$. We
write $B(U,W)=(\on_UW)^\perp$ for the second fundamental form and $H=\mathrm{tr}\,B$. Since
$\la\on_UW,\xi\ra=-\la W,\on_U\xi\ra=\la W,\Phi U\ra=0$ for $U,W\in T\Sigma$, $B$ takes values in
$\Phi(T\Sigma)$ and $\on_UW$ has no $\xi$-component. For a local orthonormal frame $\{e_i\}$ of
$T\Sigma$ the coefficients $h_{ijk}=\la B(e_i,e_j),\Phi e_k\ra$ are totally symmetric. The
$\Phi(T\Sigma)$-component of $\on_U(\Phi W)$ is $\Phi\nabla_UW$, and we write $\nabla B$ for the covariant
derivative of $B$ with respect to this connection on $\Phi(T\Sigma)$; thus $|\nabla^mB|=|\nabla^mh|$. (With the
full normal connection of $\Phi(T\Sigma)\oplus\R\xi$ one would obtain in addition the component
$\la(\nabla^\perp_{e_l}B)(e_i,e_j),\xi\ra=h_{ijl}$.) The \emph{mean curvature form} $\sigma_H=\frac12\iota_H d\eta|_{T\Sigma}$ is
closed in an $\eta$-Einstein manifold \cite[Lemma~2.8]{s1}; Smoczyk's \emph{pseudo-Einstein} condition, $\mathrm{Ric}=Kg$ on $\D$,
is the $\eta$-Einstein condition. If $\sigma_H$ is exact, there is a function $\theta$, the \emph{Legendrian angle}
\cite[Definition~3.4]{s1}, with $d\theta=\sigma_H$, that is $H=\Phi\nabla\theta$, and $|\nabla\theta|=|H|$; this is the case when $\Sigma$ is simply
connected, in particular for the Legendrian graphs of Section~\ref{sec-graph}.

A normal variation $\Phi W+h\xi$ of a Legendrian immersion preserves the Legendrian condition
to first order if and only if $W=-\frac12\nabla h$ \cite[Lemma~3.1]{s1}, \cite[Lemma~2.5]{chwu}; with
$H=\Phi\nabla\theta$ this forces $h=-2\theta$ up to a function of time, which is \eqref{c}. Along
\eqref{c} in an $\eta$-Einstein manifold \cite[Corollary~3.3]{s1}, \cite[(3.2) and Corollary~2.10]{chwu},
\begin{equation}\label{d}
\frac{\p\theta}{\p t}=\Delta\theta+(K+2)\theta ,\qquad \frac{\p}{\p t}d\mu=-|H|^2d\mu .
\end{equation}
Throughout, a solution of \eqref{c} is understood together with a smooth function $\theta$ on
$\Sigma\times[0,T)$ with $H=\Phi\nabla\theta$ which satisfies \eqref{d}; given $\theta$ at $t=0$, this determines $\theta$,
and for closed $\Sigma$ with exact mean curvature form \eqref{c} has a smooth solution for a short time,
because the normalised flow \eqref{c-norm} below has one \cite[Remark~3.6]{s1} and \eqref{c} differs from it by a rotation
along the Reeb flow. (For Sasaki--Einstein manifolds whose metric cone carries a holomorphic volume form,
\cite[Theorem~5]{dhw}, with the outline of its proof in \cite[Section~5]{dhw}, gives a short-time solution, for compact
Legendrian submanifolds of zero Maslov class, that is, with exact mean curvature form, of a flow which differs from
\eqref{c} by a rotation along the Reeb flow. This applies to $M_k$ after the $\mathcal D$-homothety of Lemma~\ref{lem-M}:
$M_k$ is simply connected, being an $S^{2k+1}$-bundle over $\mathbb{CP}^k$ (Remark~\ref{rem-topology}), so the Ricci-flat
K\"ahler cone over its Sasaki--Einstein structure is simply connected and has holonomy in $\mathrm{SU}(2k+1)$.) Replacing $\theta$ by $\theta+c(t)$ in \eqref{c} composes the solution with the Reeb flow, which
preserves the Sasakian structure; \eqref{d} allows exactly the replacements $c(t)=c_0e^{(K+2)t}$. The \emph{normalised flow}
\begin{equation}\label{c-norm}
\frac{\p F}{\p t}=H-2(\theta-\overline\theta)\,\xi,\qquad
\overline\theta(t)=\frac{1}{|\Sigma_t|}\int_{\Sigma_t}\theta\,d\mu ,
\end{equation}
therefore differs from \eqref{c} by a rotation along the Reeb flow, and it also preserves the
Legendrian condition. It is the flow \cite[(3.11)]{s1}: Smoczyk's velocity is $-2\beta X_\lambda-\nabla^k\beta\,JF_k$, where his
Legendrian pseudo-angle $\beta$ is the function with $d\beta=\sigma_H$ and $\int_\Sigma\beta\,d\mu=0$, that is $\beta=\theta-\overline\theta$, and
$J=-\Phi$ in our conventions.

\begin{remark}[The non-exact case]\label{rem-nonexact}
If $\sigma_H$ is closed but not exact, there is no Legendrian angle and \eqref{c} is not defined, but Smoczyk's flow
is. Write $\sigma_H=d\beta+\vartheta$ with $\vartheta$ harmonic for the induced metric and $\int_\Sigma\beta\,d\mu=0$; the flow is
$\p_tF=\Phi\nabla\beta-2\beta\xi$ \cite[Definitions~3.4 and~3.5]{s1}. By \cite[Lemma~3.1]{s1}, every Legendrian deformation is
generated by a function $f$, with velocity $2f\xi-\Phi\nabla f$ in our sign convention (Smoczyk's $2fX_\lambda+\nabla^kf\,JF_k$), and by
\cite[(3.10)]{s1} the volume changes by $\int\la\sigma_H,df\ra\,d\mu=\int\la d\beta,df\ra\,d\mu$, since $\vartheta$ is $L^2$-orthogonal to
exact forms. Hence Smoczyk's flow is the gradient flow of the volume on Legendrian immersions for the Dirichlet
inner product $\int\la df_1,df_2\ra\,d\mu$ on generating functions, taken modulo constants, which generate the
volume-neutral rotations along the Reeb flow, $\frac{d}{dt}|\Sigma_t|=-\int|\nabla\beta|^2d\mu$ \cite[(3.13)]{s1}. Its
horizontal velocity is $H-\Phi\vartheta^\sharp$, so that it differs from the mean curvature flow by non-local terms
\cite[Remark~3.6]{s1}; it preserves the cohomology class of $\sigma_H$ \cite[Corollary~3.3]{s1}, in particular the
exact and the non-exact cases do not mix; and its stationary solutions are the Legendrian submanifolds with harmonic
mean curvature form, not only the minimal ones (for curves, the Legendre curves of constant curvature of
\cite[Remark~4.3 and Theorem~1.1]{s1}). This paper concerns the exact case, which contains the Legendrian graphs
of Section~\ref{sec-graph}. In the non-exact case the maximum principle for $\theta$, on which the angle bounds of
Section~\ref{sec-blowup} rest, is not available for $\beta$, whose evolution contains the time derivative of the
harmonic part, and the projected flow is the mean curvature flow with the non-local forcing term $-\Phi\vartheta^\sharp$.
\end{remark}

\subsection{\texorpdfstring{$\mathcal D$}{D}-homotheties and the leaf space}\label{sec-Dhom}
For $\kappa>0$ the $\mathcal D$-homothetic structure \cite[(1.1) and Lemma~2.1]{tan} is
\begin{equation}\label{Dhom}
\eta_\kappa=\kappa\eta,\quad \xi_\kappa=\kappa^{-1}\xi,\quad \Phi_\kappa=\Phi,\quad
g_\kappa=\kappa g+(\kappa^2-\kappa)\eta\otimes\eta .
\end{equation}
It is again Sasakian with $d\eta_\kappa=2g_\kappa(\cdot,\Phi\cdot)$. Its transverse metric is
$\kappa g^T$ with the same transverse Levi-Civita connection, so $\mathrm{Ric}^T$ is unchanged and
the structure is $\eta$-Einstein with $K_\kappa+2=(K+2)/\kappa$ if $(M,\eta,g)$ is. For a Legendrian
immersion the induced metric becomes $\kappa g$, the vector $B(U,W)$ is unchanged, and
$|B|_\kappa^2=\kappa^{-1}|B|^2$, $H_\kappa=\kappa^{-1}H$, $\theta_\kappa=\theta$. Hence if $F_t$ solves
\eqref{c} for $(\eta,\xi,\Phi,g)$, then $F_{t/\kappa}$ solves \eqref{c} for the structure
\eqref{Dhom}, since $\kappa^{-1}(H-2\theta\xi)=H_\kappa-2\theta\xi_\kappa$.

A Sasakian manifold is \emph{regular} if the Reeb flow generates a free circle action. The leaf
space $Z$ is then a K\"ahler manifold, and $\pi:(M,g)\to(Z,g_Z)$ is a Riemannian submersion
with $\pi^*g_Z=g^T$ and $d\eta=\pi^*\omega_Z$, where $\omega_Z=2g_Z(\cdot,J\cdot)$ and $J$ is the
complex structure induced by $\Phi$ on $\D\cong\pi^*TZ$. We denote by $A_UW=\frac12[U,W]^v$ the
O'Neill tensor on horizontal vector fields; $A_UW=-\frac12d\eta(U,W)\xi$.

\begin{lemma}\label{lem-proj}
Let $M$ be regular and $F:\Sigma\to M$ Legendrian. Then $\pi\circ F$ is an isometric Lagrangian
immersion into $(Z,g_Z,\omega_Z)$, whose second fundamental form is $\pi_*B$ and whose mean
curvature vector is $\pi_*H$; the coefficients $h_{ijk}$ of $B$ and of $\pi_*B$ with respect to
$\{e_i\}$ and $\{\pi_*e_i\}$ coincide. Along \eqref{c} or \eqref{c-norm}, $\pi\circ F_t$ moves by
the mean curvature flow of $(Z,g_Z)$.
\end{lemma}

\begin{proof}
$\pi_*$ is an isometry on $\D\supset T\Sigma$. For horizontal vector fields $U,W$ one has
$\on_UW=(\nabla^Z_{\pi_*U}\pi_*W)^h+A_UW$, and $A_UW=0$ when $U,W$ are tangent to $\Sigma$, since
$d\eta$ vanishes on $T\Sigma$. Hence $\on_UW$ is the horizontal lift of $\nabla^Z_{\pi_*U}\pi_*W$,
and taking normal parts gives the first assertions, because $\pi_*$ maps $\Phi(T_x\Sigma)$
isometrically onto the normal space of $\pi\circ F$ and intertwines $\Phi$ with $J$. The last one
follows from $\pi_*\xi=0$.
\end{proof}

\subsection{The Heisenberg group, the horizontal potential and self-shrinkers}
On $\R^{2n+1}$ with coordinates $(x,y,z)$, $w=x+\sqrt{-1}y$, we use the structure of
\cite{cwzz} and \cite[Section~2]{cqz}:
\[
\overline\eta=\frac12dz-\frac14\sum_i(y_idx_i-x_idy_i),\qquad \overline\xi=2\p_z,
\]
\[
\overline g=\frac14\sum_i(dx_i^2+dy_i^2)+\overline\eta\otimes\overline\eta,\qquad
\overline\Phi E_i=E_{n+i},\quad \overline\Phi E_{n+i}=-E_i,\quad \overline\Phi\,\overline\xi=0,
\]
with $E_i=2\p_{x_i}+y_i\p_z$ and $E_{n+i}=-2\p_{y_i}+x_i\p_z$. Its Reeb flow is a free proper $\R$-action
with leaf space $(\R^{2n},g_0)$, $g_0=\frac14\sum_i(dx_i^2+dy_i^2)$, and in this sense it is a Sasakian
structure over $(\R^{2n},g_0)$, with $K+2=0$; all statements about regular structures used below hold for it
with this leaf space. The transverse complex structure $J_0$ induced by $\overline\Phi$ satisfies
$J_0\p_{x_i}=-\p_{y_i}$, so that the functions $x_i-\sqrt{-1}\,y_i$ are holomorphic for $J_0$; we keep the
notation $w=(x,y)\in\R^{2n}$ for the horizontal coordinates. The contact
dilations $\delta_\lambda$ satisfy $\delta_\lambda^*\overline\eta=\lambda^2\overline\eta$, so that
$\delta_\lambda$ is an isomorphism from $(\R^{2n+1},\overline\eta_\kappa,\overline\xi_\kappa,
\overline\Phi,\overline g_\kappa)$, $\kappa=\lambda^2$, onto the Heisenberg structure. Their
generator is $X=\sum_i(x_i\p_{x_i}+y_i\p_{y_i})+2z\p_z$, with $\overline\eta(X)=z$; we write
$\XH=X-z\overline\xi$ for its horizontal part, and
\begin{equation}\label{f-def}
f=\frac14|\XH|^2=\frac1{16}|w|^2
\end{equation}
for the \emph{horizontal potential}; it is the function $f$ of \cite[(4.3)]{cqz}. It is constant along the Reeb orbits, and
\begin{equation}\label{f-props}
\overline\nabla f=\frac12\XH,\qquad |\overline\nabla f|^2=f,\qquad
\overline\nabla^2f(U,U)=\frac12|U|^2\quad\text{for horizontal }U .
\end{equation}
In terms of $f$, the horizontal part of \eqref{shrinker} reads $H=-(\overline\nabla f)^\perp$. The vertical translations are
$\tau_b(x,y,z)=(x,y,z+b)$; they are automorphisms of the structure and
$\delta_\lambda\tau_b=\tau_{\lambda^2b}\delta_\lambda$.

\begin{lemma}\label{lem-horizontal}
For a Legendrian immersion into $\R^{2n+1}$ one has $(\overline\Phi\,\overline\nabla f)^T=\frac14\nabla z$.
Consequently the horizontal part $H=-(\overline\nabla f)^\perp$ of \eqref{shrinker} holds if and only if
$d(4\theta-z)=0$; and \eqref{shrinker} holds if and only if, in addition, $4\theta=z$.
\end{lemma}

\begin{proof}
From the definitions, $\overline\Phi\,\overline\nabla f=\frac12\overline\Phi\XH=\frac14\sum_i(x_iE_{n+i}+y_iE_i)$ while
$dz(E_i)=y_i$ and $dz(E_{n+i})=x_i$; hence $\overline g(\overline\Phi\,\overline\nabla f,U)=\frac14dz(U)$ for every
horizontal $U$, which is the first assertion. Since $\overline\Phi$ maps $\Phi(T\Sigma)$ onto $T\Sigma$ and $T\Sigma$
onto $\Phi(T\Sigma)$, $(\overline\Phi\,\overline\nabla f)^T=\overline\Phi((\overline\nabla f)^\perp)$. Applying $\overline\Phi$ to
$H=\overline\Phi\nabla\theta$ and to $-(\overline\nabla f)^\perp$ gives $-\nabla\theta$ and $-\frac14\nabla z$
respectively, which proves the second assertion. The $\overline\xi$-component of \eqref{shrinker}
is $-2\theta=-\frac12\overline\eta(X)=-\frac12z$.
\end{proof}

\begin{lemma}\label{lem-ss-flow}
Let $F_s$, $s<0$, be a connected Legendrian flow in $\R^{2n+1}$ solving \eqref{c} and satisfying
$H=\frac1{s}(\overline\nabla f)^\perp$ for every $s$. Assume that the induced metrics are complete and that $|B|$ and $\theta$
are bounded on $\Sigma\times I$ for every compact interval $I\subset(-\infty,0)$. Then there are $c\in\R$ and a Legendrian
self-shrinker $S$ such that, up to tangential reparametrisation, $F_s=\tau_{-4c}\big(\delta_{\sqrt{-s}}(S)\big)$.
\end{lemma}

\begin{proof}
By Lemma~\ref{lem-horizontal}, applied after the dilation $\delta_{(-s)^{-1/2}}$ (which preserves
$\theta$), $\theta+\frac{z}{4s}=c(s)$ on $F_s$. The normal velocity is therefore
\[
H-2\theta\overline\xi=\frac{1}{s}(\overline\nabla f)^\perp+\frac{z}{2s}\overline\xi-2c(s)\overline\xi
=\frac1{2s}X^\perp-2c(s)\overline\xi ,
\]
that is, $\p_sF_s=V_s^\perp$ for the ambient field $V_s=\frac{1}{2s}X-4c(s)\p_z$, since $\overline\xi=2\p_z$ is normal to
$F_s$. The function $c$ is determined by \eqref{d}, which reads $\p_s\theta=\Delta\theta$ since
$K+2=0$. By the proof of Lemma~\ref{lem-horizontal} and $dz(\overline\xi)=2$, $\overline\nabla z=4\overline\Phi\,\overline\nabla f+2\overline\xi$. For
horizontal $U,V$, \eqref{sas} gives $\overline\nabla^2z(U,V)=4\overline g(\overline\Phi\,\overline\nabla_U\overline\nabla f,V)-2\overline g(\overline\Phi U,V)$, since
$\overline\nabla f$ is horizontal, and $\overline g(\overline\Phi\,\overline\nabla_U\overline\nabla f,V)=-\overline\nabla^2f(U,\overline\Phi V)=\frac12\overline g(\overline\Phi U,V)$ by
\eqref{f-props}, polarised; hence $\overline\nabla^2z$ vanishes on horizontal vectors. As $T\Sigma$ is horizontal,
$\Delta z=\mathrm{tr}_\Sigma\overline\nabla^2z+dz(H)=dz(H)$, while $\p_sz=dz(H-2\theta\overline\xi)=dz(H)-4\theta$; thus
$(\p_s-\Delta)z=-4\theta$. Substituting $\theta=c(s)-\frac{z}{4s}$ in $(\p_s-\Delta)\theta=0$ gives
$0=c'(s)+\frac{\theta}{s}+\frac{z}{4s^2}=c'(s)+\frac{c(s)}{s}$, i.e. $c(s)=\frac{c}{-s}$ with $c=c(-1)$. The flow $\Psi_s$ of $V_s$ with $\Psi_{-1}=\mathrm{id}$ integrates $\frac{dw}{ds}=\frac{w}{2s}$ and
$\frac{dz}{ds}=\frac{z+4c}{s}$, so that $\Psi_s=\tau_{-4c}\circ\delta_{\sqrt{-s}}\circ\tau_{4c}$. Write $V_s^T=dF_s(Y_s)$ for a vector field $Y_s$ on
$\Sigma$, so that $\p_sF_s=V_s\circ F_s-dF_s(Y_s)$. Let $I\subset(-\infty,0)$ be a compact interval containing $-1$. As
$\p_sg=-2\la B,H\ra$ and $|B|$ is bounded on $\Sigma\times I$, the metrics $g_s$, $s\in I$, are uniformly equivalent, and they are
complete. Since $F_s$ is an isometric immersion, $|Y_s|_{g_s}=|V_s^T|\leq|V_s|\leq\frac{1}{2|s|}\big(\frac12|w|+|z|\big)+2|c(s)|$, by
$X=X^{\mathcal H}+z\overline\xi$, $|X^{\mathcal H}|=\frac12|w|$ and $|\p_z|=\frac12$. On $\Sigma$, $w\circ F_s$ is $2$-Lipschitz for $g_s$, because $|dw|\leq2$
on unit horizontal vectors, and $|z|=4|s||c(s)-\theta|\leq4|c|+4|s|\sup_{\Sigma\times I}|\theta|$. Hence $|Y_s|_{g_{-1}}\leq C_I\big(1+d_{g_{-1}}(x,x_0)\big)$
on $\Sigma\times I$ for a fixed $x_0\in\Sigma$, and, since $g_{-1}$ is complete, the flow $\chi_s$ of $Y_s$ with $\chi_{-1}=\mathrm{id}$ exists on $I$
and consists of diffeomorphisms of $\Sigma$; as $I$ is arbitrary, it exists for all $s<0$. Then
\[
\frac{\p}{\p s}\big(\Psi_s^{-1}\circ F_s\circ\chi_s\big)=d\Psi_s^{-1}\big(\p_sF_s-V_s\circ F_s+dF_s(Y_s)\big)\circ\chi_s=0 ,
\]
so that $F_s\circ\chi_s=\Psi_s\circ F_{-1}=\tau_{-4c}\big(\delta_{\sqrt{-s}}(S)\big)$ with $S=\tau_{4c}(F_{-1})$, on which $4\theta=z$; $S$ is a
self-shrinker by Lemma~\ref{lem-horizontal}. Completeness cannot be dropped: a fixed open disc in a Legendrian
plane through the origin satisfies the other hypotheses, with $\theta=0$, but is not of this form.
\end{proof}

\subsection{The manifolds \texorpdfstring{$M_k$}{Mk} and strict contactomorphisms}\label{sec-graph}
Let $k\geq1$, let $\eta_0$ be the standard contact form on the unit sphere $\mathbb S^{2k+1}\subset\C^{k+1}$,
with Reeb flow $u\mapsto e^{\sqrt{-1}t}u$, and let $\mathbb S^1$ act on $\mathbb S^{2k+1}\times\mathbb S^{2k+1}$ by
$(u,v)\mapsto(e^{\sqrt{-1}t}u,e^{\sqrt{-1}t}v)$. The form $\eta_0\oplus(-\eta_0)$ is invariant and vanishes
on the generator of this action, so it descends to a contact form $\eta$ on the quotient
$M_k=(\mathbb S^{2k+1}\times\mathbb S^{2k+1})/\mathbb S^1$, of dimension $4k+1$. The unit sphere with $\eta_0$ and the
round metric is Sasakian in the normalisation \eqref{sas}, with leaf space $\mathbb{CP}^k$ carrying the
Fubini--Study metric $g_{\mathrm{FS}}$ of holomorphic sectional curvature $4$, for which
$\mathrm{Ric}=2(k+1)g_{\mathrm{FS}}$. We write $J_1,J_2$ for the complex structures of the two factors of
$Z=\mathbb{CP}^k\times\mathbb{CP}^k$ and $\omega_1,\omega_2$ for the pull-backs of the K\"ahler form of $\mathbb{CP}^k$,
in the normalisation of Section~\ref{sec-Dhom}.

\begin{lemma}\label{lem-M}
With $g=\eta\otimes\eta+\pi^*(g_{\mathrm{FS}}\oplus g_{\mathrm{FS}})$, $M_k$ is a regular Sasakian manifold over
$Z$, with $d\eta=\pi^*(\omega_1-\omega_2)$ and transverse complex structure $J_1\oplus(-J_2)$. It is
$\eta$-Einstein with $K+2=2(k+1)$, and not Sasaki--Einstein. The $\mathcal D$-homothetic structure with
$\kappa=\frac{k+1}{2k+1}$ is Sasaki--Einstein.
\end{lemma}

\begin{proof}
The Reeb flow of $\eta$ is the action of $\frac12(\xi_1-\xi_2)$, which is free on $M_k$ with leaf space
$Z$, and $d\eta$ descends to $\omega_1-\omega_2$, a K\"ahler form for $J_1\oplus(-J_2)$ and
$g_{\mathrm{FS}}\oplus g_{\mathrm{FS}}$. By the correspondence between regular Sasakian structures and
K\"ahler metrics on the leaf space with the connection form $\eta$, $g$ is Sasakian with transverse
metric $g_{\mathrm{FS}}\oplus g_{\mathrm{FS}}$, whose Ricci tensor is $2(k+1)(g_{\mathrm{FS}}\oplus g_{\mathrm{FS}})$.
Thus $K+2=2(k+1)$, whereas $\dim M_k=4k+1$ requires $K=4k$ for the Sasaki--Einstein condition. By
Section~\ref{sec-Dhom}, $K_\kappa+2=2(k+1)/\kappa=4k+2$ for $\kappa=\frac{k+1}{2k+1}$.
\end{proof}

For $k=1$ the Sasaki--Einstein structure of Lemma~\ref{lem-M} is the homogeneous structure $T^{1,1}$ on
$M_1\cong\mathbb S^2\times\mathbb S^3$. Since \eqref{c} is covariant under $\mathcal D$-homotheties up to a constant
rescaling of time (Section~\ref{sec-Dhom}), every result below holds for both structures.

\begin{remark}\label{rem-topology}
The projection $[u,v]\mapsto[u]$ exhibits $M_k$ as the unit sphere bundle of $E=L^{\oplus(k+1)}$ over
$\mathbb{CP}^k$, where $L$ is the line bundle associated with the Hopf fibration. For $k\geq2$, $M_k$ is
not diffeomorphic to $\mathbb{CP}^k\times\mathbb S^{2k+1}$. Indeed, let $x\in H^2(\mathbb{CP}^k;\mathbb Z)$ be a generator.
From $TM_k\oplus\R=\pi^*(T\mathbb{CP}^k\oplus E)$ and the Pontryagin classes
$p_1(\mathbb{CP}^k)=(k+1)x^2$ and $p_1(E)=c_1(E)^2-2c_2(E)=(k+1)x^2$ \cite[Corollary~15.5 and Example~15.6]{ms} one obtains
$p_1(M_k)=2(k+1)\,\pi^*x^2$, while $p_1(\mathbb{CP}^k\times\mathbb S^{2k+1})=(k+1)x^2$. By the Gysin sequence,
$\pi^*:H^j(\mathbb{CP}^k;\mathbb Z)\to H^j(M_k;\mathbb Z)$ is an isomorphism for $j\leq2k$, and a diffeomorphism
would map $x$ to $\pm\pi^*x$; since $x^2\neq0$ and $H^4(M_k;\mathbb Z)\cong\mathbb Z$, the first Pontryagin classes
cannot correspond.
\end{remark}

A Legendrian immersion $F:\Sigma^{2k}\to M_k$ of a closed manifold is a \emph{Legendrian graph} if
$\mathrm{pr}_1\circ\pi\circ F:\Sigma\to\mathbb{CP}^k$ is a diffeomorphism. Then $F$ is an embedding and $\pi(F(\Sigma))$ is
the graph of the map $\mathrm{pr}_2\circ\pi\circ F\circ(\mathrm{pr}_1\circ\pi\circ F)^{-1}$; the requirement excludes, for
instance, the union of two different Reeb translates of a graph, which has the same projection. In this
situation the dimensions are $\dim\Sigma=n=2k$ and $\dim M_k=2n+1$.

\begin{proposition}\label{pro-graph}
If $\varphi$ is a strict contactomorphism of $(\mathbb S^{2k+1},\eta_0)$, then
$\Sigma_\varphi=\{(u,\varphi(u))\}/\mathbb S^1$ is a Legendrian graph in $M_k$, and $\pi(\Sigma_\varphi)$ is the graph
of the symplectomorphism $\tilde\varphi$ of $(\mathbb{CP}^k,\omega_{\mathrm{FS}})$ induced by $\varphi$. Conversely,
every Legendrian graph is of the form $\Sigma_\varphi$ for a unique strict contactomorphism $\varphi$.
\end{proposition}

\begin{proof}
A strict contactomorphism preserves the Reeb field, hence commutes with the Hopf action, so
that its graph $\Gamma_\varphi$ is invariant under the diagonal action; $(\eta_0\oplus-\eta_0)
|_{\Gamma_\varphi}=\eta_0-\varphi^*\eta_0=0$, so $\Sigma_\varphi$ is Legendrian, and it projects to the graph of
$\tilde\varphi$; the first projection $\Sigma_\varphi\to\mathbb{CP}^k$ is a diffeomorphism because $u\mapsto(u,\varphi(u))$ is
equivariant. Conversely, the preimage $\tilde\Sigma$ of a Legendrian graph in
$\mathbb S^{2k+1}\times\mathbb S^{2k+1}$ is an invariant $(2k+1)$-manifold on which $\eta_0\oplus(-\eta_0)$
vanishes. Its projection to the first factor is an equivariant map of principal $\mathbb S^1$-bundles over
$\mathbb{CP}^k$ covering the diffeomorphism $\mathrm{pr}_1\circ\pi|_\Sigma$, hence a diffeomorphism; thus $\tilde\Sigma$ is the
graph of a smooth map $\varphi$ with $\varphi^*\eta_0=\eta_0$. Then $\varphi^*(\eta_0\wedge(d\eta_0)^k)=\eta_0\wedge(d\eta_0)^k$, so $\varphi$ is a
local diffeomorphism of the compact connected manifold $\mathbb S^{2k+1}$, hence a covering map, and hence a
diffeomorphism because $\mathbb S^{2k+1}$ is simply connected.
\end{proof}

\begin{lemma}\label{lem-graph-pres}
Let $F:\Sigma\times[0,T)\to M_k$ be a smooth family of Legendrian immersions such that $F_0$ is a Legendrian
graph and $\mathrm{pr}_1\circ\pi\circ F_t$ is a local diffeomorphism for every $t$. Then every $F_t$ is a Legendrian
graph.
\end{lemma}

\begin{proof}
A local diffeomorphism between closed manifolds of the same dimension is a covering map. The degree of
$\mathrm{pr}_1\circ\pi\circ F_t$ is continuous in $t$ and equals $1$ at $t=0$, so each $\mathrm{pr}_1\circ\pi\circ F_t$ is a covering of
degree one, that is, a diffeomorphism.
\end{proof}

\subsection{Symplectic singular values and pinching}\label{sec-sv}
Let $\Sigma=\Sigma_\varphi$ and $x\in\Sigma$ with $\pi(x)=(q,\tilde\varphi(q))$. Since $\tilde\varphi$ is symplectic, there are
orthonormal bases $a_1,\dots,a_{2k}$ of $T_q\mathbb{CP}^k$ and $b_1,\dots,b_{2k}$ of $T_{\tilde\varphi(q)}\mathbb{CP}^k$ with
\[
a_{2l}=J_1a_{2l-1},\qquad b_{2l}=J_2b_{2l-1},\qquad d\tilde\varphi(a_i)=\lambda_ib_i,\qquad
\lambda_{2l-1}\lambda_{2l}=1
\]
\cite[Proposition~1 and Corollary~2]{mw}. The numbers $\lambda_{2l-1}$, $l=1,\dots,k$, one from each reciprocal pair, determine
all singular values. We write $i'=i+(-1)^{i+1}$ for $i=1,\dots,2k$, and use the orthonormal frame of the
horizontal lift of $\pi_*(T_x\Sigma)$, together with its image under $\Phi$,
\begin{equation}\label{frame}
e_i=\frac{(a_i,\lambda_ib_i)}{\sqrt{1+\lambda_i^2}},\qquad \Phi e_i=\frac{(J_1a_i,-\lambda_iJ_2b_i)}{\sqrt{1+\lambda_i^2}},
\qquad i=1,\dots,2k ,
\end{equation}
in which $B(e_i,e_j)=\sum_ph_{ijp}\Phi e_p$. Following \cite[Definition~1]{mw}, $\tilde\varphi$ is
\emph{$\Lambda$-pinched} if $\Lambda^{-2}g_{\mathrm{FS}}\leq\tilde\varphi^*g_{\mathrm{FS}}\leq\Lambda^2g_{\mathrm{FS}}$, that is, if all
$\lambda_i\in[\Lambda^{-1},\Lambda]$ at every point; we call $\varphi$ $\Lambda$-pinched if $\tilde\varphi$ is. Since
the Hopf projection $\mathbb S^{2k+1}\to\mathbb{CP}^k$ is a Riemannian submersion and $d\varphi$ maps the contact
distribution of $\eta_0$ to itself, covering $d\tilde\varphi$, this is the pinching of $d\varphi$ on the contact
distribution with respect to the round metric.

\section{The Jacobian of Legendrian graphs}\label{sec-jac}

\subsection{The evolution equation}

Let $\tilde J_1$ be the horizontal lift to $\D$ of the complex structure $J_1\oplus0$ of the first factor
of $Z=\mathbb{CP}^k\times\mathbb{CP}^k$, and define the basic $2$-form
\[
\overline\omega(U,W)=g(\tilde J_1U,W)\quad(U,W\in\D),\qquad \iota_\xi\overline\omega=0 .
\]
It is the pull-back of the unit-normalised K\"ahler form of the first factor, so it is closed. For a
Legendrian submanifold $F:\Sigma^{2k}\to M_k$ we define the \emph{Jacobian}
\[
\Omega=2^k\,\frac{F^*(\overline\omega^k/k!)}{d\mu},
\]
with respect to the orientation for which $\Omega>0$ on Legendrian graphs. In the frame \eqref{frame},
$\overline\omega^k/k!$ evaluated on $e_1,\dots,e_{2k}$ is the determinant of their first components in the basis
$a_1,\dots,a_{2k}$, so that
\begin{equation}\label{Omega-sv}
\Omega=2^k\prod_{i=1}^{2k}\big(1+\lambda_i^2\big)^{-1/2}=\prod_{l=1}^k\frac{2\lambda_{2l-1}}{1+\lambda_{2l-1}^2}=\prod_{l=1}^ks_l^{-1},
\qquad s_l=\frac12\big(\lambda_{2l-1}+\lambda_{2l-1}^{-1}\big)\geq1 .
\end{equation}
Thus $0<\Omega\leq1$ on Legendrian graphs, with $\Omega=1$ exactly where $d\tilde\varphi$ is an isometry, and a
Legendrian submanifold is locally a graph where $\Omega>0$. Up to the factor $2^k$, $\Omega$ is the Jacobian of
the projection of $\pi(\Sigma)$ to the first factor, the quantity $*\Omega$ of \cite{mw}. For $k=1$ one may
equally use the lift of $J_1\oplus J_2$, since $F^*\omega_1=F^*\omega_2$ on Legendrian surfaces.

\begin{lemma}\label{lem-omega}
\begin{itemize}
\item[(i)] $(\on_U\overline\omega)(V,W)=0$ for all horizontal $U,V,W$.
\item[(ii)] For $e_i,e_j,e_l$ tangent to a Legendrian submanifold and $\nu\in\Phi(T\Sigma)$,
$\la\overline R(e_l,e_i)e_j,\nu\ra=\la R^Z(\pi_*e_l,\pi_*e_i)\pi_*e_j,\pi_*\nu\ra$.
\end{itemize}
\end{lemma}

\begin{proof}
(i) For horizontal $U,V,W$, $(\on_U\overline\omega)(V,W)=(\nabla^Z\omega)(\pi_*U;\pi_*V,\pi_*W)-\overline\omega(A_UV,W)
-\overline\omega(V,A_UW)=0$, since the K\"ahler form of the first factor is parallel on $Z$, $A_UV$ is vertical and
$\iota_\xi\overline\omega=0$. (ii) Each correction term in O'Neill's formula \cite[Theorem~2]{one} for
$\la\overline R(e_l,e_i)e_j,\nu\ra$ contains one of $A_{e_l}e_i$, $A_{e_i}e_j$, $A_{e_l}e_j$, and these vanish because
$d\eta$ vanishes on $T\Sigma$. The curvature tensor of \cite{one} is the negative of ours, which plays no role here.
\end{proof}

\begin{proposition}\label{prop:jacobian}
Along \eqref{c} or \eqref{c-norm} in $M_k$, at every point where $\Omega>0$,
\begin{equation}\label{5}
\Big(\frac{\p}{\p t}-\Delta\Big)\Omega=Q(\lambda,h)\,\Omega+4\,\Omega\sum_{l=1}^k\big(1-s_l^{-2}\big),
\end{equation}
where, in the frame \eqref{frame},
\begin{equation}\label{Q}
Q(\lambda,h)=\sum_{i,j,p}h_{ijp}^2-2\sum_p\sum_{i<j}(-1)^{i+j}\lambda_i\lambda_j\big(h_{i'ip}h_{j'jp}-h_{i'jp}h_{j'ip}\big).
\end{equation}
\end{proposition}

\begin{proof}
Write $\varpi=2^k\overline\omega^k/k!$ (so that $\Omega=F^*\varpi/d\mu$), a closed basic $2k$-form annihilating $\xi$, whose covariant derivative
vanishes on horizontal arguments by Lemma~\ref{lem-omega}(i). All vectors below are horizontal.

\emph{Time derivative.} Let $V=H-2\theta\xi$ (or $H-2(\theta-\overline\theta)\xi$). Since $\varpi$ is closed and
$\iota_\xi\varpi=0$,
\[
\frac{\p}{\p t}F^*\varpi=F^*(\mathcal L_V\varpi)=F^*(d\iota_V\varpi)=d\big(F^*\iota_H\varpi\big);
\]
the vertical part of the velocity does not contribute. Evaluating $d(F^*\iota_H\varpi)$ on an orthonormal frame,
the terms containing $B$ cancel by the symmetry of $B$, the term containing $\on\varpi$ vanishes, and the
tangential part $-\sum_lH_ph_{ilp}e_l$ of $\on_{e_i}H$, $H_p=\la H,\Phi e_p\ra$, contributes $-|H|^2\Omega$, which
cancels with $\frac{\p}{\p t}d\mu=-|H|^2d\mu$. Hence
\[
\frac{\p\Omega}{\p t}=\sum_i\varpi\big(e_1,\dots,\underset{(i)}{\nabla^\perp_{e_i}H},\dots,e_{2k}\big),
\]
where $\nabla^\perp$ is the $\Phi(T\Sigma)$-component of $\on$ and the vector is placed in the $i$-th slot.

\emph{Laplacian.} At a point with $\nabla_{e_l}e_i=0$, $e_l(\Omega)=\sum_i\varpi(e_1,\dots,B(e_l,e_i),\dots,e_{2k})$.
Differentiating again, the tangential part of $\on_{e_l}(B(e_l,e_i))$ is
$-\sum_m\la B(e_l,e_i),B(e_l,e_m)\ra e_m$ by the Weingarten equation and contributes $-|B|^2\Omega$, its
$\xi$-part is annihilated by $\varpi$, and its $\Phi(T\Sigma)$-part is $(\nabla_{e_l}B)(e_l,e_i)$. By the Codazzi
equation $(\nabla_{e_l}B)(e_l,e_i)=(\nabla_{e_i}B)(e_l,e_l)+(\overline R(e_l,e_i)e_l)^\perp$, and summing over $l$,
\[
\begin{aligned}
\Delta\Omega={}&\sum_i\varpi\big(\dots,\underset{(i)}{\nabla^\perp_{e_i}H+\textstyle\sum_l(\overline R(e_l,e_i)e_l)^\perp},\dots\big)\\
&+2\sum_l\sum_{i<j}\varpi\big(\dots,\underset{(i)}{B(e_l,e_i)},\dots,\underset{(j)}{B(e_l,e_j)},\dots\big)-|B|^2\Omega .
\end{aligned}
\]
Hence, with $B(e_l,e_i)=\sum_ph_{lip}\Phi e_p$,
\begin{equation}\label{jac-general}
\begin{aligned}
\Big(\frac{\p}{\p t}-\Delta\Big)\Omega={}&|B|^2\Omega-2\sum_l\sum_{i<j}\sum_{p,q}h_{lip}h_{ljq}\,
\varpi\big(\dots,\underset{(i)}{\Phi e_p},\dots,\underset{(j)}{\Phi e_q},\dots\big)\\
&-\sum_{i,l,p}\la\overline R(e_l,e_i)e_l,\Phi e_p\ra\,\varpi\big(\dots,\underset{(i)}{\Phi e_p},\dots\big).
\end{aligned}
\end{equation}
This is the formula of \cite[Proposition~3.1]{w2} for parallel forms, which remains valid here because
only horizontal arguments occur.

\emph{Algebra.} In the frame \eqref{frame} the first component of $\Phi e_p$ is
$(1+\lambda_p^2)^{-1/2}J_1a_p=\pm(1+\lambda_p^2)^{-1/2}a_{p'}$. Hence the determinants in the second term of
\eqref{jac-general} vanish unless $\{p,q\}=\{i',j'\}$, those in the third unless $p=i'$, and evaluating them gives
$\big(Q(\lambda,h)-|B|^2\big)\Omega$ for the second term, as in the proof of \cite[Proposition~2]{mw}. By
Lemma~\ref{lem-omega}(ii) the curvature term can be computed in $Z$, with the curvature
$R(U,V)W=\la V,W\ra U-\la U,W\ra V+\la JV,W\ra JU-\la JU,W\ra JV+2\la U,JV\ra JW$ of each factor; this is the
computation of \cite[Corollary~3]{mw} for the Fubini--Study metric of holomorphic sectional curvature $1$,
multiplied by $4$, and gives $4\,\Omega\sum_l(1-\lambda_{2l-1}^2)^2(1+\lambda_{2l-1}^2)^{-2}$. Finally
$(1-\lambda^2)^2(1+\lambda^2)^{-2}=1-s^{-2}$ for $s=\frac12(\lambda+\lambda^{-1})$.
\end{proof}

\begin{remark}\label{rem-angle}
The Legendrian angle does not enter \eqref{5}: the vertical part $-2\theta\xi$ of the velocity changes neither
$F^*\overline\omega$, because $\mathcal L_{f\xi}\overline\omega=d(f\iota_\xi\overline\omega)+f\iota_\xi d\overline\omega=0$ for every function
$f$, nor the induced metric, because $\mathcal L_{f\xi}g=df\otimes\eta+\eta\otimes df$ vanishes on $T\Sigma$. We stress
that $\overline\omega$ is not parallel on $M_k$: for horizontal $U,W$, $(\on_U\overline\omega)(\xi,W)=\overline\omega(\Phi U,W)$, and a
parallel $2$-form with $\iota_\xi\overline\omega=0$ on a Sasakian manifold vanishes identically. Consequently the formula of
\cite[Proposition~3.1]{w2} for parallel forms is not used as it stands: in the proof of Proposition~\ref{prop:jacobian} the
time derivative is computed with the Cartan formula, and the computation for parallel forms is used only with
horizontal arguments, on which $\on\overline\omega$ vanishes by Lemma~\ref{lem-omega}(i). The Reeb direction contributes
nothing: $\tilde J_1$ commutes with $\Phi$ on $\D$, so that $\overline\omega(\Phi U,W)+\overline\omega(U,\Phi W)=0$, and together with
$\mathcal L_\xi\overline\omega=0$ and $\on_U\xi=-\Phi U$ this gives $\on_\xi\overline\omega=0$; the components of $\on\overline\omega$ which do not
vanish are the mixed ones displayed above.
\end{remark}

\subsection{The case \texorpdfstring{$k=1$}{k=1}: a non-negative quadratic form}
For $k=1$, using $\lambda_1\lambda_2=1$, \eqref{Q} becomes $Q=\sum_p\big((h_{11p}-h_{22p})^2+4h_{12p}^2\big)=2|B|^2-|H|^2$,
and $1-s_1^{-2}=1-\Omega^2$, so that
\begin{equation}\label{5k1}
\Big(\frac{\p}{\p t}-\Delta\Big)\Omega=\big(2|B|^2-|H|^2\big)\Omega+4(1-\Omega^2)\,\Omega ,
\end{equation}
the equation of \cite[(2.2)]{w1} for the projected flow, with $c=4$.

\begin{corollary}\label{cor-graph}
Let $k=1$. On a Legendrian surface $|H|^2\leq\frac43|B|^2$, so that $Q\geq\frac23|B|^2$. Let $\Sigma_0$ be a Legendrian
graph. If $\min_{\Sigma_0}\Omega=1$, then $\Sigma_0$ is the graph of an isometry, hence totally geodesic with constant
angle, and the normalised flow is stationary. Otherwise, along the flow
\begin{equation}\label{Omega-lower}
\Omega\geq\frac{\gamma e^{4t}}{\sqrt{1+\gamma^2e^{8t}}},\qquad \frac{\gamma}{\sqrt{1+\gamma^2}}=\min_{\Sigma_0}\Omega ;
\end{equation}
in particular $\Sigma_t$ remains a Legendrian graph by Lemma~\ref{lem-graph-pres}, $\min_{\Sigma_t}\Omega\geq\min_{\Sigma_0}\Omega$,
$1-\Omega\leq\frac12\gamma^{-2}e^{-8t}$, and $(\frac{\p}{\p t}-\Delta)\Omega\geq\frac23|B|^2\Omega$.
\end{corollary}

\begin{proof}
With $h_{111}=\alpha_1$, $h_{112}=\alpha_2$, $h_{122}=\alpha_3$, $h_{222}=\alpha_4$ one has
$|B|^2=\alpha_1^2+3\alpha_2^2+3\alpha_3^2+\alpha_4^2$ and $|H|^2=(\alpha_1+\alpha_3)^2+(\alpha_2+\alpha_4)^2$, and
$(\alpha_1+\alpha_3)^2\leq\frac43(\alpha_1^2+3\alpha_3^2)$, $(\alpha_2+\alpha_4)^2\leq\frac43(3\alpha_2^2+\alpha_4^2)$; see also
\cite[proof of Theorem~A]{w1}. Since $Q\geq0$, the maximum principle and comparison with
$\frac{d}{dt}y=4y(1-y^2)$ give \eqref{Omega-lower}, whose right-hand side is increasing in $t$.
\end{proof}

\subsection{The case \texorpdfstring{$k\geq2$}{k>=2}: pinching and its preservation}
In higher dimensions the quadratic form \eqref{Q} is not signed.

\begin{example}\label{ex-negative}
Let $k=2$, $\lambda_1=\lambda_3=a$, $\lambda_2=\lambda_4=a^{-1}$, and let $h$ be the totally symmetric tensor with
$h_{123}=h_{334}=1$, $h_{114}=-1$, and with all components not obtained from these by permutation equal to
$0$, so that $|h|^2=12$. Then
\[
Q(\lambda,h)=-\frac{4\,(a^4-5a^2+1)}{a^2},
\]
which equals $12$ for $a=1$ and is negative for $a^2>\frac{5+\sqrt{21}}2$. This is consistent with the value
$\Lambda_0(2)=\frac25\sqrt{10}+\frac15\sqrt{15}\approx2.04$ of \cite[Remark~2]{mw}, beyond which $Q$ fails to be
positive definite.
\end{example}

\begin{lemma}[{\cite[Proposition~3]{mw}}]\label{lem-mw}
For every $k\geq2$ there is $\Lambda_0(k)>1$ such that $Q(\lambda,h)\geq0$ whenever $\lambda_i\in[\Lambda_0^{-1},\Lambda_0]$ for
all $i$. For every $1\leq\Lambda_1<\Lambda_0$ there is $\delta>0$ such that $Q(\lambda,h)\geq\delta|h|^2$ whenever
$\lambda_i\in[\Lambda_1^{-1},\Lambda_1]$ for all $i$.
\end{lemma}

By Example~\ref{ex-negative}, and by embedding it in the first two complex directions, $\Lambda_0(k)<\infty$ for all
$k\geq2$. We fix
\[
\Lambda_2=\tfrac12\big(1+\Lambda_0(k)\big),\qquad S_2=\tfrac12\big(\Lambda_2+\Lambda_2^{-1}\big),\qquad
c_0(k)=S_2^{-1},
\]
let $\delta=\delta(k)$ be the constant of Lemma~\ref{lem-mw} for $\Lambda_1=\Lambda_2$, and define $\Lambda_*(k)>1$ by
$\big(\frac12(\Lambda_*+\Lambda_*^{-1})\big)^k=S_2$; compare \cite[Corollary~5]{mw}.

\begin{lemma}\label{lem-log}
For $S>1$ and $s\in[1,S]$, $1-s^{-2}\geq c_S\log s$ with $c_S=(1-S^{-2})/\log S>0$.
\end{lemma}

\begin{proof}
With $s=e^u$, $(1-s^{-2})/\log s=(1-e^{-2u})/u=\int_0^2e^{-uv}dv$, which is decreasing in $u>0$.
\end{proof}

\begin{proposition}\label{prop:pinch}
Let $k\geq2$ and let $F$ be a solution of \eqref{c} or \eqref{c-norm} on $[0,T)$ starting from a Legendrian graph with
$\min\Omega\geq c_0(k)$; by \eqref{Omega-sv} this holds if $\Sigma_0=\Sigma_{\varphi_0}$ with $\varphi_0$ $\Lambda$-pinched,
$\Lambda\leq\Lambda_*(k)$. Then on $[0,T)$:
\begin{itemize}
\item[(i)] $\Sigma_t$ is a Legendrian graph, $\min_{\Sigma_t}\Omega$ is non-decreasing, and all singular values lie in
$[\Lambda_2^{-1},\Lambda_2]$;
\item[(ii)] $(\frac{\p}{\p t}-\Delta)\Omega\geq\delta|B|^2\Omega$;
\item[(iii)] $0\leq-\log\Omega\leq\big(\max_{\Sigma_0}(-\log\Omega)\big)e^{-4ct}$ with $c=c_{S_2}$, and hence
$0\leq1-\Omega\leq Ce^{-4ct}$.
\end{itemize}
\end{proposition}

\begin{proof}
If $\Omega\geq c_0$ at a point, then $s_l\leq c_0^{-1}=S_2$ for every $l$, since all $s_{l'}\geq1$; hence
$\lambda_{2l-1}\in[\Lambda_2^{-1},\Lambda_2]$. If $\varphi_0$ is $\Lambda$-pinched, then $\Omega\geq\big(\frac12(\Lambda+\Lambda^{-1})\big)^{-k}\geq c_0$ for
$\Lambda\leq\Lambda_*$. Let $t_1\leq T$ be maximal such that $\Sigma_t$ is a Legendrian graph with all singular values in
$[\Lambda_0^{-1},\Lambda_0]$ on $[0,t_1)$. There $Q\geq0$ by Lemma~\ref{lem-mw} and the last term of \eqref{5} is
non-negative, so $\min\Omega$ is non-decreasing on $[0,t_1)$ by the maximum principle; thus $\Omega\geq c_0$ and the
singular values lie in $[\Lambda_2^{-1},\Lambda_2]$, strictly inside $[\Lambda_0^{-1},\Lambda_0]$. By continuity and Lemma~\ref{lem-graph-pres}, $t_1=T$, which
proves (i) (compare \cite[Proposition~4 and Corollary~5]{mw}), and (ii) follows from Lemma~\ref{lem-mw}. For
(iii), $(\frac{\p}{\p t}-\Delta)\log\Omega=\Omega^{-1}(\frac{\p}{\p t}-\Delta)\Omega+|\nabla\log\Omega|^2\geq4\sum_l(1-s_l^{-2})
\geq4c\sum_l\log s_l=-4c\log\Omega$ by Lemma~\ref{lem-log}, since $s_l\in[1,S_2]$ (compare
\cite[Corollary~6]{mw}). Hence $u=-\log\Omega\geq0$ satisfies $(\frac{\p}{\p t}-\Delta)u\leq-4cu$, and the maximum
principle gives the bound; finally $1-\Omega\leq-\log\Omega$.
\end{proof}

\subsection{Admissible Legendrian graphs}
We call a Legendrian graph $\Sigma_0\subset M_k$ \emph{admissible} if $k=1$, or if $k\geq2$ and $\Omega\geq c_0(k)$ on
$\Sigma_0$. Corollary~\ref{cor-graph} and Proposition~\ref{prop:pinch} give the following.

\begin{corollary}\label{cor-unified}
Let $F$ be a solution of \eqref{c} or \eqref{c-norm} on $[0,T)$, $T\leq\infty$, starting from an admissible Legendrian
graph. There are constants $c_*,\delta_*,\gamma_*>0$ and $C$, depending only on $k$ and $\min_{\Sigma_0}\Omega$, such that
on $[0,T)$ the submanifolds $\Sigma_t$ are Legendrian graphs and
\[
\Omega\geq c_*,\qquad \Big(\frac{\p}{\p t}-\Delta\Big)\Omega\geq\delta_*|B|^2\Omega,\qquad 1-\Omega\leq Ce^{-\gamma_*t};
\]
one may take $c_*=\min_{\Sigma_0}\Omega$, $\delta_*=\frac23$, $\gamma_*=8$ for $k=1$, and $c_*=c_0(k)$, $\delta_*=\delta(k)$,
$\gamma_*=4c_{S_2}$ for $k\geq2$. Moreover the singular values of $d\tilde\varphi_t$ lie in $[\Lambda_*'^{-1},\Lambda_*']$ with
$\Lambda_*'$ depending only on $c_*$.
\end{corollary}

\section{Contact blow-up analysis}\label{sec-blowup}

\subsection{Contact parabolic rescalings}\label{sec-rescale}
Let $T<\infty$ be the maximal time of a solution of \eqref{c} with $\Sigma$ closed; then
$\sup_{\Sigma_t}|B|\to\infty$ as $t\to T$ \cite[Theorem~1.1(2)]{chwu}. A point $p\in M$ is a
\emph{singular point} if there are $x_k\in\Sigma$ and $t_k\to T$ with $F(x_k,t_k)\to p$ and
$|B|(x_k,t_k)\to\infty$. The singularity is of \emph{Type I} if $\sup_{\Sigma_t}|B|^2\leq C(T-t)^{-1}$.

\begin{lemma}\label{lem-normal}
Let $M$ be a Sasakian manifold and $p\in M$. There are coordinates $(x,y,z)$ centred at
$p$, defined on $\{|w|<r,\ |z|<r\}$, in which $\xi=\overline\xi=2\p_z$ and
\[
\eta-\overline\eta=O(|w|^3)\,|dw|,\qquad g-\overline g=O(|w|^2),\qquad \Phi-\overline\Phi=O(|w|^3),
\]
with error terms independent of $z$ whose $k$-th derivatives are $O(|w|^{3-k})$, $O(|w|^{2-k})$,
$O(|w|^{3-k})$ for $k$ up to the respective orders and bounded beyond. Consequently, with
$\kappa=\lambda^2$, the structures $\sigma_\lambda=(\delta_\lambda)_*(\eta_\kappa,\xi_\kappa,\Phi,g_\kappa)$ on
$\{|w|<\lambda r,\ |z|<\lambda^2r\}$ converge to $(\overline\eta,\overline\xi,\overline\Phi,\overline g)$ in
$C^\infty_{\rm loc}(\R^{2n+1})$ as $\lambda\to\infty$.
\end{lemma}

\begin{proof}
If $M$ is not regular, let $Z$ denote the local leaf space of a flow box of the Reeb foliation around $p$, which
carries the transverse K\"ahler structure; only a chart of $Z$ is used below. Let $q_0=\pi(p)$ and choose real coordinates
$w=(x,y)$ on $Z$ centred at $q_0$ such that the functions
$x_i-\sqrt{-1}\,y_i$ are holomorphic for the complex structure $J$ induced by $\Phi$ (so that $J=J_0$ with
$J_0\p_{x_i}=-\p_{y_i}$, as in the Heisenberg model) and $g_Z=g_0+O(|w|^2)$ (K\"ahler normal coordinates);
then $\omega_Z=2g_Z(\cdot,J\cdot)=\omega_0+O(|w|^2)$ with $\omega_0=\frac12\sum dx_i\wedge dy_i$. Over this
chart $\pi$ is trivial; choose a local section through $p$ and let $z$ be the fibre coordinate with
$\xi=2\p_z$ and $z(p)=0$. Since $\eta(\p_z)=\frac12$ and $\mathcal L_\xi\eta=0$, $\eta-\frac12dz=\pi^*\beta$
with $d\beta=\omega_Z$, and replacing the section by another one changes $\beta$ by an exact form.
We choose the section so that $\beta$ is the radial primitive
$\beta_w=\int_0^1t\,\omega_Z(tw)(w,\cdot)\,dt$. As $\omega_Z(tw)=\omega_0+O(t^2|w|^2)$,
$\beta=\frac12\iota_w\omega_0+O(|w|^3)=-\frac14\sum(y_idx_i-x_idy_i)+O(|w|^3)$, which is the estimate
for $\eta$. The horizontal lift of a vector $v$ tangent to the chart is $v-2\beta(v)\p_z$, and
$\Phi$ is the horizontal lift of $J_0$; comparing with the Heisenberg lifts, which use the
quadratic part of $\beta$, gives the estimate for $\Phi$, and $g=\eta\otimes\eta+\pi^*g_Z$ gives the
one for $g$.

Under $\delta_\lambda$, $(\delta_\lambda^{-1})^*(\lambda^2\eta)=\frac12dz+\lambda^2(\delta_\lambda^{-1})^*\beta$. The
quadratic part of $\beta$ is invariant, while a term of order $|w|^3|dw|$ is multiplied by
$\lambda^2\cdot\lambda^{-4}=\lambda^{-2}$, and its derivatives likewise; similarly
$(\delta_\lambda^{-1})^*(\kappa\pi^*g_Z)=g_Z(w/\lambda)(dw,dw)=g_0+O(\lambda^{-2}|w|^2)$ and
$(\delta_\lambda^{-1})^*(\kappa\eta)^{\otimes2}\to\overline\eta\otimes\overline\eta$. Since $\Phi_\kappa=\Phi$ and
$\Phi$ is the lift of $J_0$ with respect to the rescaled connection form, it converges as well.
\end{proof}

Let $(p,T)$ be as in Theorem~\ref{thm-typeI}. For $\lambda$ large, the \emph{contact parabolic rescaling}
of scale $\lambda$ is the flow
\[
F^\lambda_s=\delta_\lambda\circ F_{T+s/\lambda^2},\qquad -\lambda^2T\leq s<0,
\]
written in the coordinates of Lemma~\ref{lem-normal} on the part of $\Sigma$ mapped into the chart.
By Section~\ref{sec-Dhom} and the equivariance under $\delta_\lambda$, $F^\lambda$ solves \eqref{c} for the
structure $\sigma_\lambda$, with Legendrian angle $\theta^\lambda(x,s)=\theta(x,T+s/\lambda^2)$, second
fundamental form $|B^\lambda|^2(x,s)=\lambda^{-2}|B|^2(x,T+s/\lambda^2)$, which is at most $C(-s)^{-1}$ under a
Type I bound, and transverse Einstein constant $K_{\lambda^2}+2=(K+2)\lambda^{-2}$. The same definition applies to every point
$p$ and every flow, without the Type I hypothesis.

\begin{lemma}\label{lem-Omega-inv}
Let $M=M_k$ and let the coordinates of Lemma~\ref{lem-normal} be built from product holomorphic normal
coordinates $w=(w_1,w_2)\in\C^k\times\C^k$ on $Z=\mathbb{CP}^k\times\mathbb{CP}^k$. Let $\overline\omega_\lambda=(\delta_\lambda)_*(\lambda^2\overline\omega)=(\delta_\lambda^{-1})^*(\lambda^2\overline\omega)$, the form associated with the
$\mathcal D$-homothetic structure with $\kappa=\lambda^2$, transported by $\delta_\lambda$, and $\Omega^\lambda$ the Jacobian of $F^\lambda_s$ with respect
to $\sigma_\lambda$ and $\overline\omega_\lambda$. Then $\Omega^\lambda(x,s)=\Omega(x,T+s/\lambda^2)$, and $\overline\omega_\lambda$ converges in
$C^\infty_{\rm loc}$ to the horizontal lift of the constant form $g_0(J_0\cdot,\cdot)$ of the first factor $\C^k$.
Moreover, $\Omega^\lambda\neq0$ at a point if and only if the projection of the tangent plane to the $w$-coordinates is
a graph over the $w_1$-plane.
\end{lemma}

\begin{proof}
Under the $\mathcal D$-homothety the transverse metric and $\overline\omega$ are multiplied by $\kappa$, so
$\overline\omega^k/k!$ is multiplied by $\kappa^k$ on $\D$, while the volume element of the induced metric of a
Legendrian submanifold of dimension $2k$ is multiplied by $\kappa^k$; hence $\Omega$ is unchanged, and $\delta_\lambda$ is an
isomorphism of the structures involved. In product holomorphic coordinates the complex structure of
the first factor is constant, and the lift of the corresponding form converges as $\Phi$ does in
Lemma~\ref{lem-normal}. The last assertion is the linear algebra preceding \eqref{Omega-sv}.
\end{proof}

\subsection{The transverse Gaussian density and its monotonicity}

\emph{The model case.} The following is the contact counterpart of Huisken's monotonicity formula
\cite[Theorem~3.1]{hu}, see \cite[Theorem~3.2]{nev11}. In the Heisenberg group, with the horizontal potential $f$ of \eqref{f-def}, let
\[
\rho_T=\big(4\pi(T-t)\big)^{-n/2}\exp\Big(-\frac{f}{T-t}\Big).
\]

\begin{proposition}\label{prop:mono-flat}
Along a solution of \eqref{c} in the Heisenberg group, with $\tau=T-t$,
\[
\frac{\p}{\p t}\big(\rho_T\,d\mu\big)=-\Big(\Delta\rho_T+\Big|H+\frac{(\overline\nabla f)^\perp}{\tau}\Big|^2\rho_T\Big)d\mu .
\]
In particular $\frac{d}{dt}\int\rho_T\,d\mu=-\int|H+\frac{1}{\tau}(\overline\nabla f)^\perp|^2\rho_T\,d\mu$ if the submanifolds $\Sigma_t$ are
complete, $|H|$ is bounded on compact time intervals, and the push-forward $\nu_t$ of $\mu_t$ under
$(x,y,z)\mapsto(x,y)$, which counts all heights and multiplicities, satisfies $\nu_t(B_R)\leq C(1+R)^N$ locally uniformly
in $t$; then the integrals converge, the time derivative may be taken under the integral, and $\int\Delta\rho_T\,d\mu=0$
by a cut-off argument.
\end{proposition}

\begin{proof}
Since $f$ is basic, $\la\overline\nabla f,\overline\xi\ra=0$, so the vertical velocity $-2\theta\overline\xi$
does not change $f$ along the flow: $\p_tf=\la\overline\nabla f,H\ra$. The tangent vectors of $\Sigma$ are
horizontal, so by \eqref{f-props} $\Delta f=\mathrm{tr}_\Sigma\overline\nabla^2f+\la\overline\nabla f,H\ra=\frac n2+
\la\overline\nabla f,H\ra$, and $|\nabla f|^2=|\overline\nabla f|^2-|(\overline\nabla f)^\perp|^2=f-|(\overline\nabla f)^\perp|^2$.
Hence
\[
\begin{aligned}
\frac{\p\rho_T}{\p t}+\Delta\rho_T-|H|^2\rho_T
&=\rho_T\Big(\frac{n}{2\tau}-\frac{f}{\tau^2}-\frac{\la\overline\nabla f,H\ra}{\tau}+\frac{|\nabla f|^2}{\tau^2}
-\frac{\Delta f}{\tau}-|H|^2\Big)\\
&=-\rho_T\Big|H+\frac{(\overline\nabla f)^\perp}{\tau}\Big|^2,
\end{aligned}
\]
which together with $\p_td\mu=-|H|^2d\mu$ is the assertion.
\end{proof}

Polynomial area growth of $\Sigma_t$ in $\R^{2n+1}$ does not suffice, since $\rho_T$ does not decay in $z$. For $n=1$, the
Legendrian helix $F(u,t)=\big(r\cos u,\ r\sin u,\ -\frac{r^2}2u\big)$, $r(t)^2=r(0)^2-8t$, solves \eqref{c} with $\theta=-u$,
is properly embedded and has linear area growth in Euclidean coordinate balls and intrinsically (quadratic growth in
balls of the Heisenberg metric, whose distance grows like $\sqrt{|z|}$ in the vertical direction), but $f=r^2/16$ is
constant on it and $\int\rho_T\,d\mu=\infty$; its angle is unbounded.

\begin{remark}\label{rem-horizontal-kernel}
Only the horizontal part of $X$ can be used. The function $\frac14|X|^2=f+\frac14z^2$ has
$\overline\nabla^2(\frac14|X|^2)(U,U)=\frac12|U|^2+\frac12dz(U)^2$ on horizontal vectors, and it is not
constant along the Reeb orbits, so that the vertical velocity $-2\theta\overline\xi$ enters its evolution; the
corresponding expression is not a square. Likewise, the Riemannian Gaussian kernel, defined through an
isometric embedding of $M$ into a Euclidean space, sees the vertical velocity, through an error term involving
$\theta^2$, and the Riemannian rescalings to which it is adapted have Lagrangian limits; in this paper it is used
only for the area bounds of Lemma~\ref{lem-area-ratio}. By Lemma~\ref{lem-horizontal}, equality
in Proposition~\ref{prop:mono-flat} is exactly the horizontal part $H=\frac{1}{s}(\overline\nabla f)^\perp$, $s=t-T$, of
the self-shrinker equation; its vertical part is recovered in the proof of Theorem~\ref{thm-typeI}.
\end{remark}

\emph{The curved case.} Let now $M$ be a compact regular Sasakian manifold. Fix an isometric
embedding $\iota:Z\to\R^N$, put $y=\iota\circ\pi\circ F$, and for $p\in M$ with $q_0=\iota(\pi(p))$ define the
horizontal potential centred at $p$ and the transverse kernel by
\[
f_p=\frac14\big|\iota\circ\pi-q_0\big|^2,\qquad
\rho_{p,T}=\big(4\pi(T-t)\big)^{-n/2}\exp\Big(-\frac{f_p}{T-t}\Big),\qquad
\Theta(t)=\int_{\Sigma_t}\rho_{p,T}\,d\mu_t .
\]
Both are constant along the Reeb orbits. In the coordinates of Lemma~\ref{lem-normal} centred at $p$,
$f_p=\frac14d_Z(\pi(\cdot),\pi(p))^2\big(1+O(d_Z^2)\big)=f\,\big(1+O(|w|^2)\big)$ and
$\overline\nabla f_p=\overline\nabla f+O(|w|^3)$, where $f=\frac14|\XH|^2$ and $\XH$ denotes the horizontal lift of the Euler field
$\sum_i(x_i\p_{x_i}+y_i\p_{y_i})$ of the chart; in the Heisenberg group these are the objects of the model
case. The embedding $\iota$ serves only to control the error terms; replacing $|\iota\circ\pi-q_0|^2$ by the
square of the transverse distance $d_Z(\pi(\cdot),\pi(p))$ changes $f_p$ by $O(f_p^2)$.

\begin{proposition}\label{prop:mono}
Let $\mathrm{II}$ be the second fundamental form of $\iota(Z)\subset\R^N$, $\mathcal E=-\mathrm{tr}_\Sigma\mathrm{II}$, and
$\mathcal N=\frac{1}{2(T-t)}(y-q_0)^{\perp_Z}-\frac12\mathcal E$, where $\perp_Z$ denotes the component normal to
$\iota(Z)$. Along \eqref{c} or \eqref{c-norm},
\[
\frac{d}{dt}\Theta=-\int_{\Sigma_t}\Big|H+\frac{(\overline\nabla f_p)^\perp}{T-t}\Big|^2\rho_{p,T}\,d\mu_t
-\int_{\Sigma_t}|\mathcal N|^2\rho_{p,T}\,d\mu_t+\frac14\int_{\Sigma_t}|\mathcal E|^2\rho_{p,T}\,d\mu_t .
\]
In particular $e^{-C_0t}\Theta(t)$ is non-increasing, where $C_0=\frac14\sup|\mathcal E|^2$ depends only
on $\iota$ and $n$, and $\lim_{t\to T}\Theta(t)$ exists. In the Heisenberg group, with $\iota$ the identity of
$(\C^n,g_0)$, $\mathcal E=\mathcal N=0$ and this is Proposition~\ref{prop:mono-flat}.
\end{proposition}

\begin{proof}
Let $\overline H$ be the mean curvature vector of $y(\Sigma_t)$ in $\R^N$. By Lemma~\ref{lem-proj}, $y$ moves
with velocity $\iota_*\pi_*H=\overline H+\mathcal E$, the induced metrics of $\Sigma_t$ and $y(\Sigma_t)$ agree, and
$\frac{\p}{\p t}d\mu_t=-|H|^2d\mu_t$ with $|H|^2=|\overline H+\mathcal E|^2$ and $\la\overline H,\mathcal E\ra=-|\mathcal E|^2$.
Write $\tau=T-t$ and $Y=y-q_0$. Then
\[
\frac{d}{dt}\rho=\rho\Big(\frac{n}{2\tau}-\frac{\la Y,\overline H+\mathcal E\ra}{2\tau}-\frac{|Y|^2}{4\tau^2}\Big),
\qquad
\Delta\rho=\rho\Big(\frac{|Y^T|^2}{4\tau^2}-\frac{n}{2\tau}-\frac{\la Y,\overline H\ra}{2\tau}\Big),
\]
and since $\int\Delta\rho\,d\mu=0$,
\[
\begin{aligned}
\frac{d}{dt}\Theta&=-\int\rho\Big(\frac{|Y^\perp|^2}{4\tau^2}+\frac{\la Y^\perp,\overline H+\frac12\mathcal E\ra}{\tau}
+|\overline H+\mathcal E|^2\Big)d\mu\\
&=-\int\rho\,\Big|\overline H+\frac12\mathcal E+\frac{Y^\perp}{2\tau}\Big|^2d\mu+\frac14\int|\mathcal E|^2\rho\,d\mu ,
\end{aligned}
\]
using $\la\overline H+\mathcal E,\mathcal E\ra=0$; here $\perp$ is normal to $y(\Sigma_t)$ in $\R^N$. Split this normal space
into the normal space of $y(\Sigma_t)$ inside $\iota(Z)$ and the normal space of $\iota(Z)$. Since
$\overline H=\iota_*\pi_*H-\mathcal E$, the second component of $\overline H+\frac12\mathcal E+\frac{Y^\perp}{2\tau}$ is $\mathcal N$. The first
component is $\iota_*\pi_*H+\frac1{2\tau}(Y^{T_Z})^{\perp}$, where $Y^{T_Z}$ is the component tangent to $\iota(Z)$;
as $\nabla^Zf_p=\frac12Y^{T_Z}$ and $\pi_*$ is an isometry on $\D$ intertwining the normal projections
(Lemma~\ref{lem-proj}), its length is $|H+\frac1\tau(\overline\nabla f_p)^\perp|$.
\end{proof}

\emph{Global and local objects.} The contact rescalings are defined on a chart, whereas $\Theta$ integrates over
all of $\Sigma_t$; the two must be kept apart. Fix $p\in M$ and the coordinates of Lemma~\ref{lem-normal} on
$U_p=\{|w|<r,\ |z|<r\}$, and let $W_p=\{|w|<r\}$ be the corresponding chart of $Z$, so that
$\pi^{-1}(W_p)\supset U_p$. For $t<T$ let $\nu_t$ be the push-forward, under $w\circ\pi\circ F_t$, of the measure
$\mu_t$ restricted to $(\pi\circ F_t)^{-1}(W_p)$. It is a measure on $W_p\subset\R^{2n}$ which counts all sheets of
$\Sigma_t$ over $W_p$, at all heights and with multiplicity. Its rescaling is the measure
$\nu^\lambda_s(A)=\lambda^n\nu_{T+s/\lambda^2}(\lambda^{-1}A)$ on $\{|w|<\lambda r\}$. The contact rescaling $F^\lambda_s$ of
Section~\ref{sec-rescale} is defined only on the part of $\Sigma$ mapped into $U_p$; the push-forward of its measure
under $w$ is at most $\nu^\lambda_s$, and in general it is smaller. For instance, if $\Sigma$ is a stationary Legendrian
submanifold and $p\notin\Sigma$ lies on the Reeb orbit of a point of $\Sigma$, then for $r$ small $F^\lambda_s$ is empty
while $\nu^\lambda_s\neq0$.

\begin{lemma}\label{lem-kernel}
Let $f^\lambda_p=\lambda^2f_p\circ\delta_\lambda^{-1}$ and $\rho^\lambda=(4\pi(-s))^{-n/2}e^{-f^\lambda_p/(-s)}$; both depend on $w$ only.
\begin{itemize}
\item[(i)] $\Theta(T+s/\lambda^2)=\int\rho^\lambda\,d\nu^\lambda_s+\epsilon(\lambda,s)$, where $c_r=\inf_{Z\setminus W_p}f_p>0$ and
\[
0\leq\epsilon(\lambda,s)\leq\big(4\pi(-s)\lambda^{-2}\big)^{-n/2}e^{-c_r\lambda^2/(-s)}\sup_{t}|\Sigma_t| ;
\]
moreover $\int_{\Sigma^\lambda_s}\rho^\lambda\,d\mu^\lambda_s\leq\int\rho^\lambda\,d\nu^\lambda_s$.
\item[(ii)] There are $C_1$ and $r_0>0$, depending on $\Sigma_0$, $\iota$ and $T$, such that
$\nu_t(B_\varrho(w_0))\leq C_1\varrho^n$ for all $t<T$, all balls with $\varrho\leq r_0$ and centres in $W_p$; hence
$\nu^\lambda_s(B_R(w_0))\leq C_1R^n$ for $R\leq\lambda r_0$.
\item[(iii)] As $\lambda\to\infty$, $f^\lambda_p\to f$ and $\overline\nabla^\lambda f^\lambda_p\to\overline\nabla f$ in $C^\infty_{\rm loc}$, and the terms
$\mathcal E$ and $\mathcal N$ of Proposition~\ref{prop:mono}, in rescaled variables, tend to $0$ locally uniformly.
\end{itemize}
\end{lemma}

\begin{proof}
(i) The integrand of $\Theta(t)$ is a function of $\pi\circ F_t$; the part of $\Sigma_t$ over $W_p$ gives
$\int\rho^\lambda d\nu^\lambda_s$ after the change of variables $w\mapsto\lambda w$, $t=T+s/\lambda^2$, and the rest is bounded
as stated because $f_p\geq c_r$ there. The inequality holds because the push-forward of the measure of the
local immersion is at most $\nu^\lambda_s$. (ii) Proposition~\ref{prop:mono}, applied with an arbitrary centre
$q\in Z$ and final time $t+\varrho^2$ on the interval $[0,t]$, gives
$\int_{\Sigma_t}\rho_{q,t+\varrho^2}\,d\mu\leq e^{C_0T}\sup_{q',\tau}\int_{\Sigma_0}\rho_{q',\tau}\,d\mu$, which is finite
because $\Sigma_0$ is closed; since $\rho_{q,t+\varrho^2}(\cdot,t)\geq c\,\varrho^{-n}$ over the $d_Z$-ball of radius $\varrho$
about $q$, and horizontal balls in the chart are comparable with these, the bound follows. (iii) By the expansion
of $f_p$, $f^\lambda_p=f\,(1+O(\lambda^{-2}|w|^2))$ and $\overline\nabla^\lambda f^\lambda_p=\overline\nabla f+O(\lambda^{-2}|w|^3)$, with
derivatives, using Lemma~\ref{lem-normal}. The rescaled embedding $\lambda(\iota(Z)-q_0)$ converges to $T_{q_0}\iota(Z)$
with second fundamental form $O(\lambda^{-1})$, so $\mathcal E^\lambda=O(\lambda^{-1})$, and the component of $\lambda(y-q_0)$
normal to it is $O(\lambda^{-1}|w|^2)$, so $\mathcal N^\lambda\to0$.
\end{proof}

For a point $p\in M$ we call $\Theta(p,T)=\lim_{t\to T}\Theta(t)$ the \emph{transverse Gaussian density} of the flow
at $(p,T)$. It is constant along the Reeb orbits, and it is the Gaussian density at $(\iota(\pi(p)),T)$ of the
projected flow $\iota\circ\pi\circ F_t$ in $\R^N$, computed with the Euclidean kernel of $\R^N$ and the dimension $n$.

\begin{proposition}\label{prop:weighted}
Let $F:\Sigma\times[0,T)\to M_k$ be a solution of \eqref{c} or \eqref{c-norm} starting from an admissible Legendrian
graph, let $c_*,\delta_*$ be as in Corollary~\ref{cor-unified}, let $p\in M_k$, and let
$\Psi(t)=\int_{\Sigma_t}(1-\Omega)\,\rho_{p,T}\,d\mu_t$, an integral over all of $\Sigma_t$. Then
\[
\frac{d\Psi}{dt}\leq-\delta_*c_*\int_{\Sigma_t}|B|^2\rho_{p,T}\,d\mu_t+C_0\,\Theta(t),
\]
with $C_0$ as in Proposition~\ref{prop:mono}. Hence $\Psi(t)-C_1t$ is non-increasing for
$C_1=C_0\sup_{[0,T)}\Theta$, and $\lim_{t\to T}\Psi(t)$ exists.
\end{proposition}

\begin{proof}
For a function $u$ on $\Sigma_t$, the computation in the proof of Proposition~\ref{prop:mono} and
$\int\rho\Delta u=\int u\Delta\rho$ give
\[
\frac{d}{dt}\int u\rho\,d\mu=\int\rho\Big(\frac{\p}{\p t}-\Delta\Big)u\,d\mu
-\int u\rho\Big(\Big|H+\frac{(\overline\nabla f_p)^\perp}{\tau}\Big|^2+|\mathcal N|^2\Big)d\mu
+\frac14\int u|\mathcal E|^2\rho\,d\mu .
\]
For $u=1-\Omega\in[0,1]$ the second term is non-positive, the third is at most $C_0\Theta$, and
$(\frac{\p}{\p t}-\Delta)(1-\Omega)\leq-\delta_*|B|^2\Omega\leq-\delta_*c_*|B|^2$ by Corollary~\ref{cor-unified}. This is the
contact counterpart of \cite[Proposition~5.2]{w4} and \cite[Lemma~7]{mw}.
\end{proof}

\subsection{Type I tangent flows}\label{sec-typeI}

\begin{lemma}\label{lem-compact}
In the situation of Theorem~\ref{thm-typeI}, with $M$ regular, for every $S>1$, $R>0$ and $m\geq0$ there are constants
$C_m$, independent of $\lambda\geq\lambda_0(S,R)$, such that on the part of $F^\lambda_s$, $s\in[-S,-S^{-1}]$, lying over
$\{|w|\leq R\}$, at all heights of the chart,
\[
|\nabla^mB^\lambda|\leq C_m,\qquad |\theta^\lambda|\leq C_\theta,\qquad \mathrm{Area}\leq C_1R^n ,
\]
where $\nabla$ is the connection of Section~\ref{sec-flow} on $\Phi(T\Sigma)$ and the area is that of the rescaled
metric.
\end{lemma}

\begin{proof}
The bound for $|B^\lambda|$ is the Type I bound. By Lemma~\ref{lem-proj} the coefficients $h_{ijk}$ of $B^\lambda$ are those
of the projected immersions, which move by the mean curvature flow of $(Z,\lambda^2g_Z)$, and with the convention
of Section~\ref{sec-flow}, $|\nabla^mB^\lambda|=|\nabla^m\mathrm{II}^\lambda_Z|$, where $\mathrm{II}^\lambda_Z$ is the second fundamental
form of the projected flow and $\nabla$ its normal connection. The rescaled metrics $\lambda^2g_Z$ have curvature
$O(\lambda^{-2})$, with all derivatives, and injectivity radius $\geq c\lambda$, so the local derivative estimates of
Chen and Yin \cite[Theorem~3.2]{cy}, applied to the projected flow of the closed manifold $\Sigma$ on the time
intervals $[s-(2S)^{-1},s]$, $s\in[-S,-S^{-1}]$, on which $|\mathrm{II}^\lambda_Z|\leq C$ by the Type I bound, and on
balls of radius $1$ of the metric induced at time $s-(2S)^{-1}$, bound $|\nabla^m\mathrm{II}^\lambda_Z|$. By \eqref{d}, $|\theta|\leq C_\theta=e^{|K+2|T}\max|\theta_0|$, and
$\theta^\lambda$ takes the same values. The area bound is Lemma~\ref{lem-kernel}(i)--(ii).
\end{proof}

\begin{lemma}\label{lem-immersion-compact}
In the situation of Lemma~\ref{lem-compact}, let $\lambda_i\to\infty$ and suppose that $F^{\lambda_i}(x_i,-1)$ lies in a fixed
compact set for some $x_i$. Then, after passing to a subsequence and composing with diffeomorphisms of the
domains, $F^{\lambda_i}$ converges smoothly on compact subsets of $\R^{2n+1}\times(-\infty,0)$ to a flow
$F_\infty:\Sigma_\infty\times(-\infty,0)\to\R^{2n+1}$ of proper Legendrian immersions for the Heisenberg structure,
where $\Sigma_\infty$ is a manifold without boundary, possibly disconnected, each component of which is complete.
$F_\infty$ solves \eqref{c}, its angle $\theta_\infty$ is the limit of $\theta^{\lambda_i}$, the measures of $F^{\lambda_i}_s$ converge
to that of $F_{\infty,s}$ for every $s$, and the push-forward $\nu_{\infty,s}$ of the latter under $w$ satisfies
$\nu_{\infty,s}(B_R)\leq C_1R^n$ for all $R$.
\end{lemma}

\begin{proof}
For fixed $s$, $F^{\lambda_i}_s$ is a proper immersion into the chart $\{|w|<\lambda_ir,\ |z|<\lambda_i^2r\}$, and these
domains exhaust $\R^{2n+1}$. The rescaled metrics converge in $C^\infty_{\rm loc}$ to the Heisenberg metric, so by
Lemma~\ref{lem-compact} the second fundamental forms with respect to the coordinate metric, with all
derivatives, and the measures are bounded on every compact set, uniformly for $i$ large. The compactness
theorem for proper immersions of Breuning \cite[Theorem~1.3 and Corollary~7.13, with the weakened hypotheses
stated before it]{bre} gives a subsequence converging locally smoothly, after reparametrisation, to a proper
immersion of a manifold without boundary, and \cite[Corollary~1.4]{bre} the convergence of the measures. To obtain a single limiting flow, fix a compact interval $I\Subset(-\infty,0)$ containing $-1$, apply Breuning's
theorem at $s=-1$, and transport the resulting charts of the source to the other times of $I$ by the flow itself,
which is normally parametrised. On the relevant regions the Type I and angle bounds give $|\p_sw|\leq C_I$ and
$|\p_sz|\leq C_I(1+|w|)$, since $\p_sz=dz(H)-4\theta$ and $|dz(H)|=2|\beta(H)|\leq C|w||H|$ in the coordinates of
Lemma~\ref{lem-normal}. Hence points lying over a compact set at one time of $I$ remain in a larger compact set
throughout $I$, and, by the same estimate backwards in time, every sheet meeting a compact set at some time of $I$
is already accounted for at $s=-1$. Moreover $|\p_sg_\Sigma|_{g_\Sigma}\leq2|H||B|\leq C_I$, so the induced metrics are
uniformly comparable on the transported charts, and the spatial bounds of Lemma~\ref{lem-compact}, together with
\eqref{c}, \eqref{d} and $d\theta=\sigma_H$, bound all mixed space-time derivatives in these charts. The Arzel\`a--Ascoli
theorem and a diagonal exhaustion in space and time give a subsequence converging smoothly on compact subsets
of space-time to a flow with a single source manifold. The Legendrian condition, \eqref{c} and the convergence of the angles pass to the limit,
since the structures converge and the transverse Einstein constant $(K+2)\lambda_i^{-2}$, which is the coefficient of $\theta$
in \eqref{d}, tends to $0$. A proper immersion
into a complete manifold has a complete induced metric. The bound on $\nu_{\infty,s}$ follows from
Lemma~\ref{lem-compact}, because each compact subset of $\R^{2n+1}$ lying over $B_R$ is eventually contained in the
charts.
\end{proof}

\begin{proof}[Proof of Theorem~\ref{thm-typeI}]
We give the proof for regular $M$; the changes for arbitrary $M$ are in Section~\ref{sec-general}.

\emph{Non-emptiness.} The velocity satisfies $|H|\leq C(T-t)^{-1/2}$ and $|2\theta\xi|\leq C$, which is
integrable on $[0,T)$; hence $F(\cdot,t)$ converges uniformly to a continuous map $F(\cdot,T)$ as
$t\to T$. If $p$ is a singular point with $F(x_k,t_k)\to p$, a subsequence of $x_k$ converges to some
$x_0$, and $F(x_0,T)=p$. In the coordinates of Lemma~\ref{lem-normal}, the $w$-coordinate of
$F(x_0,t)$ satisfies $|w|\leq C(T-t)^{1/2}$, since $|\p_tw|\leq C|H|$. For the $z$-coordinate,
$dz=2\eta-2\pi^*\beta$ gives $\p_tz=-4\theta-2\beta(H)$ (for \eqref{c}; the normalised flow differs by a
bounded term), and $|\beta(H)|\leq C|w||H|\leq C$; hence $|z|\leq C(T-t)$. In rescaled coordinates the
point $F^\lambda(x_0,s)$ therefore stays in $\{|w|\leq C,\ |z|\leq C\}$ for $s\in[-2,-1]$.

\emph{Compactness.} By the first step and Lemma~\ref{lem-immersion-compact}, every sequence $\lambda_i\to\infty$ has a
subsequence along which $F^{\lambda_i}$ converges to a non-empty flow $F_\infty$ as described there.

\emph{Self-similarity of the horizontal part.} By Proposition~\ref{prop:mono}, the defect
$\int_{t_1}^{t_2}\int_{\Sigma_t}|H+\frac{(\overline\nabla f_p)^\perp}{T-t}|^2\rho_{p,T}\,d\mu\,dt$, an integral over all of $\Sigma_t$, is at most
$\Theta(t_1)-\Theta(t_2)+C_0(t_2-t_1)\sup\Theta$. For $t_j=T+s_j/\lambda^2$ this tends to $0$ as $\lambda\to\infty$, because
$\lim_{t\to T}\Theta(t)$ exists. The part of this integral coming from the chart is, in rescaled variables,
\[
\int_{s_1}^{s_2}\!\!\int_{\Sigma^\lambda_s}\Big|H^\lambda+\frac{(\overline\nabla^\lambda f^\lambda_p)^\perp}{-s}\Big|^2\rho^\lambda\,d\mu\,ds ,
\]
which is therefore at most the full defect and tends to $0$ as well. By Lemma~\ref{lem-kernel}(iii) and Fatou's
lemma, $F_\infty$ satisfies $H=\frac{1}{s}(\overline\nabla f)^\perp$ for every $s<0$. Each connected component of $F_\infty$ satisfies
the other hypotheses of Lemma~\ref{lem-ss-flow}: its induced metrics are complete by Lemma~\ref{lem-immersion-compact},
$|B|^2\leq C(-s)^{-1}$ by the Type I bound, and $|\theta_\infty|\leq C_\theta$ by Lemma~\ref{lem-compact} (Lemma~\ref{lem-compact-general} for
arbitrary $M$). By Lemma~\ref{lem-ss-flow}, every connected component of $F_\infty$ is of the form $\tau_{-4c}(\delta_{\sqrt{-s}}S)$ for a Legendrian self-shrinker $S$ on
which $4\theta_\infty=z$; this proves the last assertion of the theorem, except for the finiteness of the number of
components, which we now prove.

\emph{A bound on the number of components.} Let $\tau_{-4c}(\delta_{\sqrt{-s}}S)$ be a component of $F_\infty$. Since
$z=4\theta_\infty$ on $S$ and $|\theta_\infty|\leq C_\theta$, $S$ lies in the slab $\{|z|\leq4C_\theta\}$; as $S$ is properly immersed,
its horizontal potential $f$ is proper on $S$. At a minimum point of $f$ on $S$ we have $\nabla f=0$ and $\Delta f\geq0$.
On the other hand $\Delta f=\frac n2+\la\overline\nabla f,H\ra=\frac n2-|(\overline\nabla f)^\perp|^2$ by \eqref{f-props} and
$H=-(\overline\nabla f)^\perp$, and $|\nabla f|^2=f-|(\overline\nabla f)^\perp|^2$; hence $\min_Sf\leq\frac n2$, that is, $S$ meets
$\{|w|^2\leq8n\}$. Since vertical translations do not change $w$, every component of $F_{\infty,-1}$ meets
$\{|w|^2\leq8n\}$. We claim that each component contains an embedded piece of area at least $c\,\ell_1^n$, with
$c=c(n)>0$ and $\ell_1=\ell_1(n,C)>0$, lying over $\{|w|\leq\sqrt{8n}+1\}$, where $C$ is the bound on $|B|$ at $s=-1$.
On $\{|w|\leq\sqrt{8n}+2\}$ the Heisenberg metric $\overline g$ and the Euclidean metric $g_e$ of the coordinates
$(x,y,z)$ are uniformly equivalent and the Christoffel symbols of $\overline g$ are bounded, with constants depending
only on $n$, because the coefficients of $\overline g$ depend on $w$ only. Hence the second fundamental form $A$ of the
component with respect to $g_e$ is bounded there by a constant $C_2=C_2(n,C)$: since
$\overline\nabla_UW=\nabla_UW+B(U,W)$ and $\overline\nabla-\nabla^e$ is a bounded tensor, $A(U,W)$ is the $g_e$-normal
part of $B(U,W)-(\overline\nabla-\nabla^e)(U,W)$. Consequently, along a curve in the component of $g_e$-length $\ell$
the tangent plane, as a point of the Grassmannian with the distance $|P-P'|$ given by the operator norm of the
difference of the orthogonal projections, moves by at most $2C_2\ell$. Let $x_0$ be a point of the component with
$|w(x_0)|^2\leq8n$, let $P_0$ be its tangent plane, let $\ell_1=\min\{1,(16C_2)^{-1}\}$, and let $V$ be the set of
points of the component joined to $x_0$ by a curve of $g_e$-length less than $\ell_1$. Every tangent plane $P$ at a
point of $V$ satisfies $|P-P_0|\leq\frac18$, so for a unit vector $v\in P$ the orthogonal projection $\Pi$ onto $P_0$
satisfies $|\Pi v|\geq\frac78$ and $|v-\Pi v|\leq\frac18$; thus $\Pi$ restricted to $V$ is a local diffeomorphism
onto its image whose inverse differential has norm at most $2$. Every segment in $P_0$ starting at $\Pi(x_0)$ of
length less than $\ell_1/4$ therefore lifts uniquely to a curve in $V$ starting at $x_0$: the lift has $g_e$-length at
most twice that of the segment, hence less than $\ell_1/2$, so it stays in $V$, and by the completeness of the
component it can be continued up to the end of the segment. These lifts define a smooth section $\sigma$ of $\Pi$
over the disc $D$ of radius $\ell_1/4$ in $P_0$ about $\Pi(x_0)$; $\sigma(D)$ is an embedded piece of the component,
the graph over $D$ of a function with gradient of norm at most $\frac17$, and its points are at $g_e$-distance less
than $\ell_1\leq1$ from $x_0$, so it lies over $\{|w|\leq\sqrt{8n}+1\}$. Its $g_e$-area is at least the area of $D$,
hence its area for $\overline g$ is at least $c\,\ell_1^n$ with $c=c(n)$, which proves the claim. By
Lemma~\ref{lem-immersion-compact}, $\nu_{\infty,-1}(\{|w|\leq\sqrt{8n}+2\})\leq C_1(\sqrt{8n}+2)^n$, and
$\nu_{\infty,-1}$ counts the pieces of different components, and the copies of a component of higher multiplicity,
separately. Hence $F_\infty$ has at most $N_0=C_1(\sqrt{8n}+2)^n/(c\,\ell_1^n)$ components, counted with multiplicity,
and $N_0$ depends only on the constants of Lemma~\ref{lem-compact}.

\emph{The vertical centre.} Let $\mathcal T$ be the set of all limits obtained in this way along sequences
$\lambda_i\to\infty$, with the topology of smooth convergence on compact subsets as immersions with multiplicity. It
is closed, by a diagonal argument, since all its elements satisfy the bounds of Lemma~\ref{lem-immersion-compact}
with the same constants. If $F_\infty\in\mathcal T$ is the limit along $\lambda_i$, then $D_\sigma F_\infty$, defined by
$(D_\sigma F_\infty)_s=\delta_\sigma\circ(F_\infty)_{s/\sigma^2}$, is the limit along $\sigma\lambda_i$, because
$F^{\sigma\lambda}_s=\delta_\sigma\circ F^\lambda_{s/\sigma^2}$, the chart at scale $\sigma\lambda$ is the image under $\delta_\sigma$ of the chart at scale
$\lambda$, and $\sigma_{\sigma\lambda}=(\delta_\sigma)_*(\sigma_\lambda)_{\sigma^2}$; hence $\mathcal T$ is invariant under $D_\sigma$, $\sigma>0$. Since
$\delta_\sigma\tau_b=\tau_{\sigma^2b}\delta_\sigma$, the finitely many components $\tau_{-4c_j}(\delta_{\sqrt{-s}}S_j)$ of $F_\infty$ are mapped by
$D_\sigma$ to $\tau_{-4c_j\sigma^2}(\delta_{\sqrt{-s}}S_j)$, and as $\sigma\to0$ these converge smoothly on compact subsets, with the
same multiplicities, to $\delta_{\sqrt{-s}}S_j$. Hence the flow $F^0$ with components $\delta_{\sqrt{-s}}S_j$ lies in $\mathcal T$, that
is, it is the limit along a sequence $\lambda_i\to\infty$. The topology of $\mathcal T$ does not record the angles, so we
identify the limit angle along this sequence. By Lemma~\ref{lem-immersion-compact}, after passing to a subsequence, the
angles converge to a function $\theta^0$ on the source of $F^0$, and the steps above apply to $F^0$ with the angle $\theta^0$:
each component of $F^0$ is of the form $\tau_{-4c'}(\delta_{\sqrt{-s}}S')$ for some $c'\in\R$ and a Legendrian
self-shrinker $S'$ on which $4\theta^0=z$. This component is also $\delta_{\sqrt{-s}}S_j$ for some $j$. At $s=-1$ this gives
$S'=\tau_{4c'}(S_j)$, and then, since $\delta_\sigma\tau_b=\tau_{\sigma^2b}\delta_\sigma$, $S_j=\tau_{4c'(1+1/s)}(S_j)$ for every $s<0$. If
$c'\neq0$, then $S_j$ is invariant under a non-trivial vertical translation and hence under all its iterates, which
is impossible because $S_j$ is non-empty and lies in the slab $\{|z|\leq4C_\theta\}$. Hence $c'=0$, so that $S'=S_j$ and
$4\theta^0=z$ on $S_j$. By Lemma~\ref{lem-horizontal}, applied after the dilation $\delta_{(-s)^{-1/2}}$, the flow $F^0$ with the
limit angle $\theta^0$ satisfies $H-2\theta^0\xi=\frac1{2s}X^\perp$, and it is non-empty.
\end{proof}

\begin{remark}\label{rem-gauge}
The angle $\theta_\infty$ is the limit of the angle of the flow in the gauge fixed by \eqref{d}; it is a function on
$\Sigma_\infty$ with $H=\Phi\nabla\theta_\infty$, and $4\theta_\infty=z$ on each $S_j$. The Legendrian angle $\theta^{\rm cqz}$ of \cite{cqz} is
defined on an oriented Legendrian submanifold with the holomorphic volume form of the Heisenberg group
\cite[(2.14)]{cqz}, reversing the orientation changes it by $\pi$, and its local branches satisfy
$H=\Phi\nabla\theta^{\rm cqz}$ \cite[(2.15)]{cqz}. Hence, if $S_j$ is oriented, the local branches of $\theta^{\rm cqz}$ differ from
$\theta_\infty$ by locally constant functions; since $S_j$ is connected, $e^{\sqrt{-1}(\theta_\infty-\theta^{\rm cqz})}$ is constant, and
$\theta^{\rm cqz}=\theta_\infty-a_j$ is a single-valued branch for a constant $a_j$. These constants may differ between components;
a single ambient motion need not normalise all of them. The vertical translations are left translations, which
preserve $\theta^{\rm cqz}$, and on $\tau_{-4a_j}(S_j)$ we have $z-4a_j=4\theta_\infty-4a_j=4\theta^{\rm cqz}$; by \cite[Proposition~3.3(i)]{cqz},
$\tau_{-4a_j}(S_j)$ is a Legendrian self-shrinker in the sense of \cite[(3.4)]{cqz}, with Legendrian angle $\theta_\infty-a_j$. We use
this normalisation component by component. That every $S_j$ is orientable is shown in the proof of
Corollary~\ref{cor-planes}.
\end{remark}

\begin{corollary}\label{cor-planes}
In Theorem~\ref{thm-typeI}, every tangent flow is a union of finitely many Legendrian $n$-planes, counted with
multiplicity.
\end{corollary}

\begin{proof}
By the proof of Theorem~\ref{thm-typeI}, each component of a tangent flow is a vertical translate of
$\delta_{\sqrt{-s}}S_j$, where $S_j$ is connected, complete, properly immersed and lies in a slab, and carries the function
$\theta_\infty$, with $H=\Phi\nabla\theta_\infty$, $4\theta_\infty=z$ and $|\theta_\infty|\leq C_\theta$. We first show that $S_j$ is orientable. Otherwise
let $\pi:\tilde S_j\to S_j$ be its orientation double cover, which is connected and oriented. Since $\pi$ is a finite covering,
the immersion of $\tilde S_j$ is again complete and proper, and $\theta_\infty\circ\pi$ satisfies $H=\Phi\nabla(\theta_\infty\circ\pi)$,
$4\,\theta_\infty\circ\pi=z$ and $|\theta_\infty\circ\pi|\leq C_\theta$. Remark~\ref{rem-gauge} therefore applies to $\tilde S_j$: for a constant $\tilde a$,
$\tau_{-4\tilde a}(\tilde S_j)$ is a complete properly immersed Legendrian self-shrinker in the sense of \cite[(3.4)]{cqz} with bounded
Legendrian angle, and by \cite[Theorem~A]{cqz} its image is a Legendrian $n$-plane $P$. The immersion of the
$n$-manifold $\tilde S_j$ onto the $n$-plane $\tau_{4\tilde a}(P)$ is then a proper local diffeomorphism, hence a covering map,
and it is injective because the plane is simply connected and $\tilde S_j$ is connected. This contradicts its invariance
under the non-trivial deck transformation of $\pi$. Hence $S_j$ is orientable, and we orient it. By
Remark~\ref{rem-gauge}, $\tau_{-4a_j}(S_j)$ is a complete properly immersed Legendrian self-shrinker in the sense of
\cite[(3.4)]{cqz}, with bounded Legendrian angle; by \cite[Theorem~A]{cqz} it is a Legendrian $n$-plane, and hence so are
$S_j$ and its translates and dilates.
\end{proof}

\subsection{Local regularity and the density at Type I singular points}

\begin{lemma}\label{lem-eps}
Let $M$ be a compact regular Sasakian manifold with an isometric embedding $\iota:Z\to\R^N$, and let
$\varepsilon>0$ be White's constant for flows of $n$-dimensional submanifolds of $\R^N$ \cite[Theorem~3.1 and
Proposition~3.2]{wh}. Let $F:\Sigma\times[0,T)\to M$ be a solution of \eqref{c} with $\Sigma$ closed, and $p\in M$. If
$\Theta(p,T)<1+\varepsilon$, then there are a neighbourhood $U$ of $\pi(p)$ and $\delta>0$ such that $|B|$ is bounded on
$\{x:\pi(F(x,t))\in U\}$ for $t\in(T-\delta,T)$.
\end{lemma}

\begin{proof}
By Lemma~\ref{lem-proj}, $\Gamma_t=\iota\circ\pi\circ F_t$ is a smooth flow of immersed $n$-dimensional submanifolds of the
compact submanifold $\iota(Z)\subset\R^N$ moving by the mean curvature of $\iota(Z)$, that is, by the mean curvature in
$\R^N$ plus the Brakke operator $-\mathrm{tr}\,\mathrm{II}$ of problem~(3) in \cite[Section~4]{wh}; and $\Theta(p,T)$ is its Gaussian
density at $X_0=(\iota(\pi(p)),T)$. We follow the proof of \cite[Theorem~3.5 and Corollary~4.2]{wh}. Choose
$\varepsilon'<\varepsilon$ with $\Theta(p,T)<1+\varepsilon'$. By definition there is $r>0$ such that the Gaussian density ratio at $X_0$
at the fixed scale $r$ is less than $1+\varepsilon'$; this ratio is continuous in the centre and the time at the fixed
scale, so it is less than $1+\varepsilon'$ at every point $Y$ of a spacetime neighbourhood $\mathcal U$ of $X_0$; and by the
almost monotonicity of Proposition~\ref{prop:mono}, applied with centres in $\mathcal U$, the ratios at $Y$ at all scales
$\rho\leq r$ are at most $e^{C_0r^2}(1+\varepsilon')<1+\varepsilon$ after decreasing $r$. At a self-intersection point of $\Gamma_t$ the ratios
tend to at least $2$ as $\rho\to0$, so $\Gamma_t$ is embedded in $\mathcal U$, and the flow is a proper flow in $\mathcal U$ in the sense
of \cite[Section~2.3]{wh}. By \cite[Theorem~4.1]{wh}, applied in $\mathcal U$ with the Lipschitz Brakke operator
$-\mathrm{tr}\,\mathrm{II}$, the flow is $C^{2,\alpha}$ with bounds on a smaller neighbourhood of $X_0$; in particular the second fundamental form
$\widehat B$ of $\iota\circ\pi\circ F_t$ in $\R^N$ is bounded there. Since
$\widehat B(U,V)=\iota_*\pi_*B(U,V)+\mathrm{II}(\pi_*U,\pi_*V)$ with orthogonal summands, where $\mathrm{II}$ is the second
fundamental form of $\iota(Z)$ and we use Lemma~\ref{lem-proj}, $|B|\leq|\widehat B|$ is bounded as well. The constant $\varepsilon$ depends only on $n$
and $N$, that is, on the dimension of $\Sigma$ and on the embedding $\iota$.
\end{proof}

\begin{proof}[Proof of Corollary~\ref{cor-typeI-intro}]
Let $p$ be a singular point, $X_0=(y_0,T)$ with $y_0=\iota(\pi(p))$. The flow $\Gamma_t=\iota\circ\pi\circ F_t$ of the proof of
Lemma~\ref{lem-eps} satisfies the same Type I bound, the Legendrian angle $\theta$ is a global potential of its mean curvature
form, $d\theta=\sigma_H$, with $|\theta|\leq C_\theta$ by \eqref{d}, and its Gaussian density at $X_0$ is $\Theta(p,T)$. (On a
K\"ahler--Einstein manifold $\theta$ need not be a phase with respect to a parallel holomorphic volume form, which
need not exist.) Consider its Riemannian
parabolic rescalings $\Gamma^\lambda_s=\lambda(\Gamma_{T+s/\lambda^2}-y_0)$ in $\R^N$, $s\in[-2,-1]$, $\lambda\geq1$, and let
$\rho(y,s)=(4\pi(-s))^{-n/2}e^{-|y|^2/(4(-s))}$. By the Type I bound and \cite[Theorem~3.2]{cy} they have locally bounded
second fundamental form with all derivatives. By the argument of Lemma~\ref{lem-kernel}(ii),
$\mathcal H^n(\Gamma^\lambda_s\cap B_R)\leq C_1R^n$ for $R\leq\lambda r_0$, where $\mathcal H^n$ counts multiplicity; for $R>\lambda r_0$ the total mass
satisfies $\mathcal H^n(\Gamma^\lambda_s)\leq\lambda^n|\Sigma_0|\leq|\Sigma_0|r_0^{-n}R^n$, since the area is non-increasing by \eqref{d}. Hence
$\mathcal H^n(\Gamma^\lambda_s\cap B_R)\leq C_2R^n$ for all $R>0$, with $C_2=\max\{C_1,|\Sigma_0|r_0^{-n}\}$, and a decomposition into dyadic
annuli, as in Step~4 of the proof of Theorem~\ref{thm-flat}, gives uniform Gaussian tails,
$\int_{\{|y|>R\}}\rho(\cdot,-1)\,d\mathcal H^n\llcorner\Gamma^\lambda_{-1}\leq Ce^{-R^2/8}$ for $R\geq1$ and all $\lambda\geq1$. In rescaled variables the
monotonicity defect of Proposition~\ref{prop:mono} at time $s$ is
$D_\lambda(s)=\int|\widehat H+\frac12\mathcal E^\lambda+\frac{y^\perp}{2(-s)}|^2\rho(\cdot,s)\,d\mathcal H^n\llcorner\Gamma^\lambda_s$, where $\widehat H$ is the mean curvature
vector in $\R^N$ and $\mathcal E^\lambda$ is $\mathcal E$ in rescaled variables, and the proof of Proposition~\ref{prop:mono} gives
\[
\int_{-2}^{-1}D_\lambda(s)\,ds\leq\Theta(T-2\lambda^{-2})-\Theta(T-\lambda^{-2})+C_0\lambda^{-2}\sup_{[0,T)}\Theta\longrightarrow0\qquad(\lambda\to\infty).
\]
Choose $s_\lambda\in[-2,-1]$ with $D_\lambda(s_\lambda)\leq\int_{-2}^{-1}D_\lambda\,ds$ and put $\lambda'=\lambda(-s_\lambda)^{-1/2}\in[\lambda/\sqrt2,\lambda]$. Then
$\Gamma^{\lambda'}_{-1}=(-s_\lambda)^{-1/2}\Gamma^\lambda_{s_\lambda}$, and by scaling $D_{\lambda'}(-1)=(-s_\lambda)D_\lambda(s_\lambda)\to0$. By the bounds above, which hold
at every scale, \cite[Theorem~1.3 and Corollary~7.13, with the weakened hypotheses stated before it]{bre} and
\cite[Corollary~1.4]{bre}, a subsequence of $\Gamma^{\lambda'}_{-1}$ converges locally smoothly, as immersions with multiplicity, to a
proper immersion of a manifold without boundary, possibly empty, into $T_{y_0}\iota(Z)\cong\C^n$, with convergence of the
measures; the limit is Lagrangian, every component is complete, and the potentials converge to $\theta_\infty$ with
$d\theta_\infty=\sigma_{H_\infty}$ and $|\theta_\infty|\leq C_\theta$. By Lemma~\ref{lem-kernel}(iii), with $\iota$, $\mathcal E^{\lambda'}\to0$ locally uniformly, so
$D_{\lambda'}(-1)\to0$ and Fatou's lemma give $H_\infty=-\frac12y^\perp$ on the limit. Let $L$ be a component of the limit: a smooth
complete properly immersed Lagrangian self-shrinker without boundary, with $d\theta_\infty=\sigma_{H_\infty}$ and
$|\theta_\infty|\leq C_\theta$ on $L$. Since the measures converge and open balls are open sets, $L$ inherits the area bound
$\mathcal H^n(L\cap B_R)\leq C_2R^n$ for all $R>0$, counted with multiplicity; in particular the Gaussian integrals over $L$
used in the proof of \cite[Corollary~3.5(i)]{nev11} are finite. On the flat limit $T_{y_0}\iota(Z)\cong\C^n$ choose a parallel complex volume form. If $L$ is
oriented, its Lagrangian angle $\theta_\Omega$ satisfies $d\theta_\Omega=\sigma_{H_\infty}$ \cite[Proposition~2.2]{ev22}, so that
$\theta_\Omega-\theta_\infty$ is constant on $L$, and $L$ is a zero-Maslov class Lagrangian self-shrinker with bounded Lagrangian
angle, hence a plane by \cite[Corollary~3.5(i)]{nev11}; these are exactly the hypotheses stressed in
\cite[Remark~3.6(b)]{nev11}. The component $L$ is orientable: otherwise its orientation double cover $\tilde L$, which is
connected and oriented, is again a smooth complete properly immersed Lagrangian self-shrinker, with
$\mathcal H^n(\tilde L\cap B_R)\leq2C_2R^n$, since the covering has two sheets, and the pull-back of $\theta_\infty$ has the same properties; hence the image of $\tilde L$ is a plane, the
immersion of $\tilde L$ onto this plane is a proper local diffeomorphism, hence a covering map, and it is injective
because the plane is simply connected, which contradicts its invariance under the non-trivial deck
transformation. Hence every component of the limit is a plane. A self-shrinking plane passes through the
origin, so its Gaussian integral is $1$. The Gaussian integral $\int\rho(\cdot,-1)\,d\mathcal H^n\llcorner\Gamma^{\lambda'}_{-1}$ equals
$\Theta(T-\lambda'^{-2})$, which tends to $\Theta(p,T)$ because $\lambda'\to\infty$; by the uniform Gaussian tails and the convergence
of the measures it also tends to the Gaussian integral of the limit. Hence $\Theta(p,T)$ equals the number of planes
of the limit, counted with multiplicity; in particular it is an integer. If it were at most $1$, Lemma~\ref{lem-eps}
would bound $|B|$ near $p$, contrary to the choice of $p$. Hence $\Theta(p,T)\geq2$.
\end{proof}

\begin{remark}\label{rem-sheets}
The contact tangent flows of Theorem~\ref{thm-typeI} see only the sheets of $\Sigma_t$ at height $O(T-t)$ above $p$ in the
coordinates of Lemma~\ref{lem-normal}, whereas the transverse density $\Theta(p,T)$, being constant along the Reeb
orbits, counts all sheets which approach the Reeb orbit of $p$; for instance $\Theta\geq2$ at a regular point whose
Reeb orbit meets two sheets of the flow, as happens at the endpoints of a Reeb chord. For this reason
Lemma~\ref{lem-eps} cannot be used to exclude Type I singularities at points where several sheets meet, and a local
regularity theorem for \eqref{c} in terms of a density localised in the Reeb direction is not available. When $K+2\neq0$ and the
projected flow is embedded, the exclusion result \cite[Theorem~3.2]{ev22} applies instead, as explained in the
introduction; it does not cover the case $K+2=0$. By
Corollary~\ref{cor-planes}, the contact tangent flows at a Type I singular point are unions of Legendrian planes lying
over Lagrangian subspaces of $\R^{2n}$. In the Lagrangian setting, Neves combined the corresponding fact for
zero-Maslov class flows with White's theorem in the same way \cite[proof of Theorem~4.1]{nev07}.
\end{remark}

\subsection{Flatness of the contact blow-ups of Legendrian graphs}

We now prove Theorem~\ref{thm-flat}; the Type I hypothesis of Theorem~\ref{thm-typeI} is replaced by
Proposition~\ref{prop:weighted}. Throughout, $n=2k$, and the coordinates of Lemma~\ref{lem-normal} are those of
Lemma~\ref{lem-Omega-inv}, with $w=(w_1,w_2)\in\R^{2k}\times\R^{2k}$.

\begin{proof}[Proof of Theorem~\ref{thm-flat}]
The initial graph is admissible, so by Corollary~\ref{cor-unified} the submanifolds $\Sigma_t$ remain Legendrian
graphs, $\Omega\geq c_*$, and the singular values of $d\tilde\varphi_t$ lie in $[\Lambda^{-1},\Lambda]$ with $\Lambda=\Lambda(c_*)$. In the
chart, the projection of $\Sigma_t$ over $\{|w_1|<r\}$ is contained in the graph $w_2=u_t(w_1)$ of a map with
$|Du_t|\leq L$, $L=L(\Lambda)$ (Lemma~\ref{lem-Omega-inv}), and the whole of $\pi(\Sigma_t)$ is the graph of $\tilde\varphi_t$. Fix
$p\in M_k$.

\emph{Step 1: decay of the curvature at all scales.} By Proposition~\ref{prop:weighted}, whose integrals are over
all of $\Sigma_t$,
\[
\delta_*c_*\int_{T-2\lambda^{-2}}^{T-\lambda^{-2}}\!\!\int_{\Sigma_t}|B|^2\rho_{p,T}\,d\mu\,dt\leq
\Psi\big(T-2\lambda^{-2}\big)-\Psi\big(T-\lambda^{-2}\big)+\frac{C_1}{\lambda^2}\longrightarrow0
\qquad(\lambda\to\infty),
\]
and the left-hand side is invariant under the rescaling $t=T+s/\lambda^2$, $|B^\lambda|^2=\lambda^{-2}|B|^2$. Choose $\lambda_j\to\infty$
and $s_j\in[-2,-1]$, $s_j\to s_\infty$, such that, with $t_j=T+s_j/\lambda_j^2$, $\lambda_j^{-2}\int_{\Sigma_{t_j}}|B|^2\rho_{p,T}(\cdot,t_j)\,d\mu\to0$.

\emph{Step 2: the relevant sheet and its height.} Either for every $R$ and all large $j$ no point of $\Sigma_{t_j}$ lies
over $\{|w|\leq R\lambda_j^{-1}\}$, or, after passing to a subsequence, there are $R_0$ and points of $\Sigma_{t_j}$ over
$\{|w|\leq R_0\lambda_j^{-1}\}$. In the second case $|u_{t_j}(0)|\leq(1+L)R_0\lambda_j^{-1}$. Fix $\rho_1\in(0,r)$ so small that
$(1+L)\rho_1<r/2$ and $2C(1+L)^2\rho_1^2<r/2$, where $C$ is the constant in $|\beta|\leq C|w|$ of Lemma~\ref{lem-normal}. For large
$j$ the part of $\pi(\Sigma_{t_j})$ over $\{|w_1|\leq\rho_1\}$ is the graph of $u_{t_j}$ over this ball and lies in $W_p$. Compose
$F_{t_j}$ with the rotation along the Reeb flow, an automorphism of the structure which changes neither $\Omega$, $|B|$
nor $\Theta$, that places the point of $\Sigma_{t_j}$ over $(0,u_{t_j}(0))$ at height $z=0$. Along the lifted sheet the Legendrian
condition reads $dz=-2\beta(dw)$, and $|w|\leq(1+L)|w_1|+|u_{t_j}(0)|$ on it; integrating along radial segments, its height over
$\{|w_1|\leq\rho_1\}$ is at most $C(1+L)^2\rho_1^2+2C(1+L)|u_{t_j}(0)|\rho_1$, which for large $j$ is at most $2C(1+L)^2\rho_1^2<r/2$,
since $u_{t_j}(0)\to0$; hence the whole sheet over this ball lies in $U_p$. After rescaling, it is the horizontal lift over the graph of
$u_j(w_1)=\lambda_ju_{t_j}(w_1/\lambda_j)$, defined on $\{|w_1|\leq\lambda_j\rho_1\}$, which exhausts $\R^{2k}$, with $|Du_j|\leq L$ and
$|u_j(0)|\leq(1+L)R_0$, and its height $z_j$ satisfies $z_j(0)=0$ and $dz_j=-2\beta^{\lambda_j}(dw_1,du_j)$, so that
$|z_j|\leq C(R)$ on $\{|w_1|\leq R\}$.

\emph{Step 3: the blow-up is a Legendrian $2k$-plane.} Fix $R$. By Step~2 the graphs of $u_j$ over $\{|w_1|\leq R\}$ lie in a
fixed compact set $K_R$ of the $w$-coordinates. Let $h_j$ be the rescaled transverse metric and
$\epsilon_j=\|h_j-g_0\|_{C^1(K_R)}$, which tends to $0$. With $G_j(w_1)=(w_1,u_j(w_1))$ and $P_j^\perp$ the $h_j$-normal
projection along the graph,
\[
B_{h_j}(\p_a,\p_b)=P_j^\perp\big((0,\p_a\p_bu_j)+\Gamma(h_j)(\p_aG_j,\p_bG_j)\big),
\]
and $B_{h_j}$ is the projected second fundamental form, of norm $|B^{\lambda_j}|$ by Lemma~\ref{lem-proj}. Since $|Du_j|\leq L$,
$\|P_j^\perp-P_0^\perp\|\leq C(L)\epsilon_j$ and $|\Gamma(h_j)|\leq C\epsilon_j$; with the flat estimate
$|D^2u|\leq(1+|Du|^2)^{3/2}|P_0^\perp(0,D^2u)|$ \cite[Lemma~2.2]{bre} (the factor $\frac14$ in $g_0$ only changes constants) this
gives
\[
|D^2u_j|\leq C(L)|B^{\lambda_j}|+C(L)\epsilon_j|D^2u_j|+C(L)\epsilon_j .
\]
For $j$ large the middle term can be absorbed, so $|D^2u_j|\leq C(L)(|B^{\lambda_j}|+\epsilon_j)$; no pointwise bound on
$D^2u_j$ is needed. The kernel $\rho^{\lambda_j}(\cdot,s_j)$ is bounded below on $K_R$, and the area element of the graph is
comparable to Lebesgue measure, so by Step~1
\[
\int_{\{|w_1|\leq R\}}|D^2u_j|^2\,dw_1\leq C_R\int_{\Sigma^{\lambda_j}_{s_j}\cap\pi^{-1}(K_R)}|B^{\lambda_j}|^2\rho^{\lambda_j}\,d\mu+C_R\epsilon_j^2\longrightarrow0 .
\] By the Arzel\`a--Ascoli theorem and the Poincar\'e inequality a subsequence of
$u_j$ converges in $C^\alpha_{\rm loc}$, with $Du_j\to Du_\infty$ in $L^2_{\rm loc}$, to a map $u_\infty$ with $D^2u_\infty=0$, that is, an
affine map. Since $\beta^{\lambda_j}\to\overline\beta=-\frac14\sum(y_idx_i-x_idy_i)$ locally uniformly, $z_j$ converges in
$C^\alpha_{\rm loc}\cap W^{1,2}_{\rm loc}$ to the solution of $dz_\infty=-2\overline\beta(dw_1,du_\infty)$ with $z_\infty(0)=0$. The limit
$\{(w_1,u_\infty(w_1),z_\infty(w_1))\}$ is the horizontal lift of the affine Lagrangian plane $\mathrm{graph}(u_\infty)$, a Legendrian
$2k$-plane of the Heisenberg group in the sense explained after Theorem~\ref{thm-flat}, of multiplicity one; $u_\infty(0)$ need
not vanish. In the first alternative of Step~2 the limit near the origin is
empty. This proves (i).

\emph{Step 4: the transverse density.} By Lemma~\ref{lem-kernel}(i),
$\Theta(t_j)=\int\rho^{\lambda_j}d\nu^{\lambda_j}_{s_j}+o(1)$, where $\nu^{\lambda_j}_{s_j}$ is the full projected measure. We first bound its
tails, in both alternatives of Step~2. By Lemma~\ref{lem-kernel}(ii), $\nu^\lambda_s(B_R)\leq C_1R^n$ for $R\leq\lambda r_0$, and for
$R>\lambda r_0$ the total mass satisfies $\nu^\lambda_s\leq\lambda^n|\Sigma_0|\leq|\Sigma_0|r_0^{-n}R^n$; hence $\nu^\lambda_s(B_R)\leq CR^n$ for all
$R\geq1$. On the chart $f^\lambda_p(w)\geq c|w|^2$, so for $s_j\in[-2,-1]$ a decomposition into dyadic annuli gives
\[
\int_{\{|w|>R\}}\rho^{\lambda_j}\,d\nu^{\lambda_j}_{s_j}\leq C\sum_{\ell\geq0}(2^{\ell+1}R)^ne^{-c\,4^\ell R^2},
\]
uniformly in $j$, and the horizontal tails are uniformly small,
\[
\lim_{R\to\infty}\sup_j\int_{\{|w|>R\}}\rho^{\lambda_j}\,d\nu^{\lambda_j}_{s_j}=0 .
\]
In the second alternative, since $\pi(\Sigma_{t_j})$ is a graph, the part of $\nu^{\lambda_j}_{s_j}$ over $\{|w_1|\leq R\}$ is the measure
of the graph of $u_j$, and there the area elements converge, because $Du_j\to Du_\infty$ in $L^2_{\rm loc}$ and $|Du_j|\leq L$.
Therefore
\[
\Theta(p,T)=\lim_{j\to\infty}\Theta(t_j)=\int_P(4\pi(-s_\infty))^{-k}e^{-f/(-s_\infty)}\,d\mu\leq1,
\]
where $P$ is the limit plane; the last integral is the Gaussian integral of the affine plane $\pi(P)$ in $(\R^{4k},g_0)$,
which is at most $1$. In the first alternative of Step~2, for every $R$ all of $\nu^{\lambda_j}_{s_j}$ lies over $\{|w|>R\}$ when $j$ is large, so
$\Theta(p,T)=0$ by the tail estimate. This proves (ii).

\emph{Step 5: regularity.} By (ii) and Lemma~\ref{lem-eps}, $|B|$ is bounded near $p$ as $t\to T$, which is (iii).
Covering $M_k$ by finitely many such neighbourhoods, $|B|$ is bounded on $\Sigma\times[0,T)$, which contradicts
$\sup_{\Sigma_t}|B|\to\infty$ \cite[Theorem~1.1(2)]{chwu}. Hence the maximal time of the flow is infinite.
\end{proof}

\begin{remark}
Theorem~\ref{thm-flat} does not use Theorem~\ref{thm-typeI}. The two statements are proved with the same rescalings and the
same transverse density; in Theorem~\ref{thm-typeI} compactness comes from the Type I bound, in Theorem~\ref{thm-flat} from
the Jacobian, which is invariant under the rescalings and turns the transverse density into the almost monotone
quantity $\Psi$ controlling $\int|B|^2$. The argument is the contact counterpart of the one of
\cite[Proposition~6.1]{w4}, \cite[Theorem~A]{w1} and \cite[Section~3.3]{mw} for graphs of symplectomorphisms. By
Lemma~\ref{lem-proj}, the long-time existence in Theorem~\ref{thm-flat} is equivalent to that of the projected flow proved
there.
\end{remark}

\subsection{Arbitrary \texorpdfstring{$\eta$}{eta}-Einstein Sasakian manifolds}\label{sec-general}

Regularity was used in Sections~\ref{sec-rescale}--\ref{sec-typeI} only through the leaf space $Z$: in the transverse density
$\Theta$, whose kernel is a function on $Z$ and which integrates over all sheets above a transverse chart
(Proposition~\ref{prop:mono} and Lemma~\ref{lem-kernel}), and in the derivative estimates of Lemma~\ref{lem-compact},
which were obtained on the projected flow. We now remove it.

Let $M$ be a compact Sasakian manifold, not necessarily regular, and $p\in M$. Near $p$ the Reeb foliation is a
fibration: there is a \emph{flow box} $U\cong W\times(-2\delta,2\delta)$, foliated by arcs of Reeb orbits, whose local
leaf space $W$ carries the transverse K\"ahler structure, and $\pi_U:U\to W$ is a Riemannian submersion. The proof
of Lemma~\ref{lem-normal} uses only a chart of the leaf space, so the coordinates of Lemma~\ref{lem-normal} exist on
$U$, with $W=\{|w|<2r\}$ and $|z|<2\delta$. The proof of Lemma~\ref{lem-proj} is pointwise, so it holds for
$\pi_U$: the parts of a Legendrian flow lying in $U$ project to Lagrangian immersions in $W$ moving by the mean
curvature flow of $W$. We take $W$ so small that it has an isometric embedding $\iota_W:W\to\R^N$ and put
$f_p=\frac14|\iota_W\circ\pi_U-\iota_W(\pi_U(p))|^2$, with $\rho_{p,T}$ as before.

\begin{lemma}\label{lem-area-ratio}
Let $M$ be a compact Sasakian manifold and $F:\Sigma\times[0,T)\to M$ a solution of \eqref{c} with $\Sigma$ closed and
$|\theta|\leq C_\theta$. There are $C_A$ and $s_0>0$ such that $\mu_t\big(F_t^{-1}(B_s(x))\big)\leq C_As^n$ for all
$x\in M$, $s\leq s_0$ and $t\in[0,T)$, where $B_s(x)$ is the geodesic ball of $M$.
\end{lemma}

\begin{proof}
Let $\iota_M:M\to\R^{N'}$ be an isometric embedding. Then $\iota_M\circ F_t$ moves with normal velocity $\widehat H+V$,
where $\widehat H$ is its mean curvature in $\R^{N'}$ and
\[
V=-\mathrm{tr}_\Sigma\mathrm{II}_M-2\theta\,\iota_{M*}\xi ,
\]
which is normal to the submanifold and bounded. For the Euclidean kernel $\rho_{y_0,T'}$ of dimension $n$,
the computation in the proof of \cite[Theorem~3.1]{hu}, which does not depend on the codimension (compare
\cite[Theorem~3.2]{nev11}), gives
\[
\frac{d}{dt}\int\rho_{y_0,T'}\,d\mu=-\int\Big|\widehat H+\frac{(y-y_0)^\perp}{2(T'-t)}\Big|^2\rho_{y_0,T'}\,d\mu
+\int\Big\la V,\,-\widehat H-\frac{(y-y_0)^\perp}{2(T'-t)}\Big\ra\rho_{y_0,T'}\,d\mu .
\]
Absorbing the second term gives $\frac{d}{dt}\int\rho_{y_0,T'}\,d\mu\leq\frac14\sup|V|^2\int\rho_{y_0,T'}\,d\mu$. Hence the Gaussian
density ratios at all centres and scales are bounded by $e^{CT}$ times their supremum at $t=0$, which is finite
because $\Sigma$ is closed. The area bound follows as in Lemma~\ref{lem-kernel}(ii).
\end{proof}

\begin{proposition}\label{prop:mono-local}
Let $\chi_{\rm h}\in C^\infty_c(W)$ and $\zeta\in C^\infty_c((-2\delta,2\delta))$ satisfy $0\leq\chi_{\rm h},\zeta\leq1$, with $\chi_{\rm h}=1$ on $\{|w|\leq r\}$
and $\zeta=1$ on $[-\delta,\delta]$. Put
\[
\Theta_{\rm loc}(t)=\int_{\Sigma_t}\chi_{\rm h}(\pi_U\circ F)\,\zeta(z\circ F)\,\rho_{p,T}\,d\mu ,
\]
with the integrand extended by $0$ outside $U$. Then, with $C_0=\frac14\sup|\mathcal E|^2$ for the embedding $\iota_W$,
\[
\begin{aligned}
\frac{d}{dt}\Theta_{\rm loc}&\leq-\int_{\Sigma_t}\chi_{\rm h}\zeta\Big(\Big|H+\frac{(\overline\nabla f_p)^\perp}{T-t}\Big|^2+|\mathcal N|^2\Big)\rho_{p,T}\,d\mu
+C_0\Theta_{\rm loc}+E(t),\\
0\leq E(t)&\leq C_E\big(1+(T-t)^{-1/2}\big),
\end{aligned}
\]
where $C_E$ depends on $\chi_{\rm h}$, $\zeta$, $C_\theta$, $C_A$, $|\Sigma_0|$ and the chart. Consequently:
\begin{itemize}
\item $\Theta_{\rm loc}$ is bounded on $[0,T)$ and $\lim_{t\to T}\Theta_{\rm loc}(t)$ exists;
\item $\int_0^T\!\!\int\chi_{\rm h}\zeta\,|H+(\overline\nabla f_p)^\perp/(T-t)|^2\rho_{p,T}\,d\mu\,dt<\infty$;
\item there are $C_1$ and $r_0>0$ such that the push-forward $\nu^{\rm box}_t$ under $w\circ\pi_U\circ F_t$ of $\mu_t$ restricted to the
part of $\Sigma_t$ in $\{|w|<r,\ |z|<\delta\}$ satisfies $\nu^{\rm box}_t(B_\varrho(w_0))\leq C_1\varrho^n$ for $\varrho\leq r_0$ and $|w_0|\leq r/2$.
\end{itemize}
\end{proposition}

\begin{proof}
The computation in the proof of Proposition~\ref{prop:mono} is pointwise, so it holds on $U$ with $\iota_W$ in place
of $\iota$. With the weighted form used in the proof of Proposition~\ref{prop:weighted} and $u=\chi_{\rm h}\zeta$, it gives
the inequality with
\[
E(t)=\int_{\Sigma_t}\Big|\Big(\frac{\p}{\p t}-\Delta\Big)(\chi_{\rm h}\zeta)\Big|\,\rho_{p,T}\,d\mu .
\]
We bound the two kinds of cut-off terms separately.
\begin{itemize}
\item \emph{Horizontal cut-off.} The derivatives of $\chi_{\rm h}$ are supported where $f_p\geq c>0$, so they contribute at most
$C(T-t)^{-n/2}e^{-c/(T-t)}|\Sigma_0|$.
\item \emph{Vertical cut-off.} In the coordinates of Lemma~\ref{lem-normal},
\[
\Big(\frac{\p}{\p t}-\Delta\Big)z=dz(H-2\theta\xi)-\mathrm{tr}_\Sigma\overline\nabla^2z-dz(H)=-4\theta-\mathrm{tr}_\Sigma\overline\nabla^2z ,
\]
since $dz(\xi)=2$. This is bounded, as is $|\nabla z|$. Hence
\[
\Big|\Big(\frac{\p}{\p t}-\Delta\Big)\zeta(z)\Big|\leq C\,\mathbf 1_{\{\delta\leq|z|\leq2\delta\}} .
\]
To bound the integral of $\rho_{p,T}$ over this band, write $\tau=T-t$, $s_\ell=2^\ell\sqrt\tau$, and let
$A_\ell=\{s_{\ell-1}\leq|w|\leq s_\ell\}$ ($A_0=\{|w|\leq\sqrt\tau\}$). The set $A_\ell\times[\delta,2\delta]$ is covered by at most
$C(1+\delta/s_\ell)$ geodesic balls of radius $s_\ell$, because on the chart the metric is uniformly equivalent to the
Euclidean metric of the coordinates. By Lemma~\ref{lem-area-ratio}, the band has area at most
$C(s_\ell^n+\delta s_\ell^{n-1})$ over $A_\ell$ when $s_\ell\leq s_0$, and at most $|\Sigma_0|$ otherwise, while
$\rho_{p,T}\leq C\tau^{-n/2}e^{-c4^\ell}$ there, since $f_p\geq c|w|^2$ on the chart. Summing over $\ell$,
\[
\int_{\Sigma_t\cap\{\delta\leq|z|\leq2\delta\}}\rho_{p,T}\,d\mu\leq C\delta\,\tau^{-1/2}\sum_{\ell\geq0}2^{\ell(n-1)}e^{-c4^\ell}+C\leq C'\big(1+\tau^{-1/2}\big).
\]
\end{itemize}
Thus $\frac{d}{dt}(e^{-C_0t}\Theta_{\rm loc})\leq e^{-C_0t}E(t)$, and $E$ is integrable on $[0,T)$. So $\Theta_{\rm loc}$ is bounded,
$e^{-C_0t}\Theta_{\rm loc}(t)-\int_0^te^{-C_0t'}E(t')\,dt'$ is non-increasing and has a limit, and the defect integral is finite. For
the last assertion, apply the same argument with centre $q$, $\pi_U(q)=w_0$, $z(q)=0$, final time $t+\varrho^2$ and the
interval $[0,t]$; the constants are uniform for $|w_0|\leq r/2$, since the localised densities of $\Sigma_0$ are
bounded uniformly in the centre and the scale ($\Sigma_0$ is compact and $\pi_U\circ F_0$ is an immersion on the
support of the cut-offs). The kernel is at least $c\varrho^{-n}$ over $B_\varrho(w_0)$, and
$\chi_{\rm h}\zeta=1$ on the part of $\Sigma_t$ counted by $\nu^{\rm box}_t$.
\end{proof}

\begin{lemma}\label{lem-compact-general}
In the situation of Theorem~\ref{thm-typeI}, for an arbitrary compact $\eta$-Einstein Sasakian manifold $M$ and with the chart of
Lemma~\ref{lem-normal} in the flow box $U$, for every $S>1$, $R>0$ and $m\geq0$ there are constants $C_m$, independent of
$\lambda\geq\lambda_0(S,R)$, with the following properties for $s\in[-S,-S^{-1}]$: $|\nabla^mB^\lambda|\leq C_m$ at the points of
$F^\lambda_s$ with $|w|,|z|\leq R$; $|\theta^\lambda|\leq C_\theta$; and the part of $F^\lambda_s$ over $\{|w|\leq R\}$ inside the rescaled
inner box $\{|w|<\lambda r,\ |z|<\lambda^2\delta\}$ has area at most $C_1R^n$. These are the bounds of Lemma~\ref{lem-compact}
that are used in Lemma~\ref{lem-immersion-compact}.
\end{lemma}

\begin{proof}
The angle bound is \eqref{d}. The area bound is the last assertion of Proposition~\ref{prop:mono-local}, rescaled. It
remains to prove the derivative bounds.

\emph{Derivative bounds.} These are obtained from the flow of the whole closed manifold $\Sigma$, without projecting
to a leaf space, so that no locality question arises. By Section~\ref{sec-Dhom}, $s\mapsto F_{T+s/\kappa}$, $\kappa=\lambda^2$,
is a solution of \eqref{c} of the closed manifold $\Sigma$ in the compact Sasakian manifold
$M_\kappa=(M,\eta_\kappa,\xi_\kappa,\Phi,g_\kappa)$, which is $\eta$-Einstein with $K_\kappa+2=(K+2)/\kappa$; $F^\lambda$ is its
restriction to the chart transported by $\delta_\lambda$, which is an isometry from $g_\kappa$ onto the metric of $\sigma_\lambda$.
The curvature tensor of $g_\kappa$ and its covariant derivatives are bounded on $M$ uniformly in $\kappa\geq1$: the
curvature tensor of a Sasakian manifold is the sum of its transverse curvature tensor and of universal terms in $g$,
$\eta$, $\xi$ and $\Phi$ (the O'Neill formulas \cite[Section~4]{one} for the Riemannian submersion onto the local leaf space, in
which $A_UW=-g(U,\Phi W)\xi$, together with $\overline R(U,W)\xi=-(\on_U\Phi)W+(\on_W\Phi)U=\eta(W)U-\eta(U)W$, which follows from \eqref{sas};
see also \cite[Proposition~7.3]{bl}), while
$\overline\nabla\eta$ and $\overline\nabla\Phi$ are universal by \eqref{sas}; the $\mathcal D$-homothety multiplies the transverse
metric by $\kappa$ and leaves the transverse connection unchanged, so that the $g_\kappa$-norms of the transverse
curvature tensor and of its transverse derivatives are those for $g$ multiplied by negative powers of $\kappa$. Let
$K_j=\sum_{i\leq j}\sup_M|\overline\nabla^i\overline{\mathrm{Rm}}|_{g_\kappa}$, which is thus bounded independently of $\kappa$.
Along \eqref{c} in $M_\kappa$ the second fundamental form satisfies the evolution inequalities \cite[(3.4)]{chwu}
\[
\begin{aligned}
\Big(\frac{\p}{\p s}-\Delta\Big)|\nabla^mB|^2\leq{}&-2|\nabla^{m+1}B|^2
+C(n,m)\Big\{\sum_{i+j+k=m}|\nabla^iB||\nabla^jB||\nabla^kB||\nabla^mB|\\
&+K_m\sum_{i\leq m}|\nabla^iB||\nabla^mB|+K_{m+1}|\nabla^mB|\Big\},
\end{aligned}
\]
whose constants depend only on $n$, $m$ and on the bounds $K_m$, $K_{m+1}$ of the ambient curvature and its
derivatives (this is how the ambient manifold enters in \cite[(3.4) and Lemma~3.8]{chwu}). In \cite{chwu} these constants
are said to depend on $n$ and the ambient manifold. In the derivation, the ambient manifold enters only through
$\overline{\mathrm{Rm}}$ and its covariant derivatives up to order $m+1$, and through the structure tensors $\xi_\kappa$,
$\eta_\kappa$, $\Phi$. The covariant derivatives of the structure tensors are universal expressions by \eqref{sas}, so
their contributions to the displayed inequality have universal coefficients. Since $\overline R(U,\xi_\kappa)\xi_\kappa=U$ for
$U\perp\xi_\kappa$, we have $K_0\geq1$, and these contributions are absorbed in the terms with $K_m$ and $K_{m+1}$. Hence the
constants are uniform in $\kappa\geq1$. Here $\nabla^mB$ is taken in the
convention of \cite{chwu}; the bounds obtained below transfer to the convention of Section~\ref{sec-flow}, since the two
connections differ by a bounded tensor, so that the two $m$-th derivatives of $B$ differ by expressions in
$B,\dots,\nabla^{m-1}B$ with bounded coefficients.

Fix $s_0\in[-S,-S^{-1}]$, put $\tau=(2S)^{-1}$ and $\sigma_j=s_0-2^{-j}\tau$ for $j\geq0$, so that
$s_0-\tau=\sigma_0<\sigma_1<\sigma_2<\dots<s_0$. For $\lambda\geq\lambda_0(S,R)$ the flow is defined on
$[\sigma_0,s_0]\subset[-S-\frac12,-S^{-1}]$, and the Type I bound gives $|B|\leq C_B$ on $\Sigma\times[\sigma_0,s_0]$, with $C_B$
depending only on $S$ and the Type I constant. We show by induction on $m\geq0$ that
\begin{equation}\tag{$\ast_m$}\label{shi-induction}
|\nabla^mB|\leq C_m\quad\text{on }\Sigma\times[\sigma_m,s_0],
\end{equation}
with $C_m$ depending only on $n$, $m$, $C_B$, $S$ and $K_{m+1}$. For $m=0$ this is the Type I bound, with $C_0=C_B$. Let
$m\geq1$ and assume $(\ast_i)$ for $0\leq i\leq m-1$; since $\sigma_i\leq\sigma_{m-1}$, these bounds hold on
$\Sigma\times[\sigma_{m-1},s_0]$. There, in the inequality displayed above with $m-1$ in place of $m$, every term on the
right-hand side except $-2|\nabla^mB|^2$ is bounded. In the inequality for $|\nabla^mB|^2$ every product contains at most
two factors $|\nabla^mB|$ and all its other factors are bounded, so that these products are at most
$C''(1+|\nabla^mB|^2)$. Hence
\[
\begin{aligned}
\Big(\frac{\p}{\p s}-\Delta\Big)|\nabla^{m-1}B|^2&\leq-2|\nabla^mB|^2+C',\\
\Big(\frac{\p}{\p s}-\Delta\Big)|\nabla^mB|^2&\leq-2|\nabla^{m+1}B|^2+C''\big(1+|\nabla^mB|^2\big)
\end{aligned}
\]
on $\Sigma\times[\sigma_{m-1},s_0]$, with $C',C''$ depending only on $n$, $m$, $C_B$, $S$ and $K_{m+1}$. Put $u=s-\sigma_{m-1}$, so
that $0\leq u\leq\tau$ on $[\sigma_{m-1},s_0]$, and $\varphi=u|\nabla^mB|^2+\beta|\nabla^{m-1}B|^2$ with $\beta=1+\tau C''$. Dropping the
term $-2u|\nabla^{m+1}B|^2$, we obtain
\[
\Big(\frac{\p}{\p s}-\Delta\Big)\varphi\leq|\nabla^mB|^2+\tau C''\big(1+|\nabla^mB|^2\big)-2\beta|\nabla^mB|^2+\beta C'
\leq\tau C''+\beta C' ,
\]
because the coefficient of $|\nabla^mB|^2$ is $1+\tau C''-2\beta=-\beta<0$. At $s=\sigma_{m-1}$ we have $u=0$ and
$\varphi\leq\beta C_{m-1}^2$, so the maximum principle on the closed manifold $\Sigma$ gives
$\varphi\leq\beta C_{m-1}^2+(\tau C''+\beta C')\tau$ on $\Sigma\times[\sigma_{m-1},s_0]$. On $[\sigma_m,s_0]$ we have
$u\geq\sigma_m-\sigma_{m-1}=2^{-m}\tau$, hence
$|\nabla^mB|^2\leq2^m\tau^{-1}\big(\beta C_{m-1}^2+(\tau C''+\beta C')\tau\big)$ there, which is $(\ast_m)$. The induction uses only the
inequalities of \cite[(3.4)]{chwu} displayed above and the maximum principle (for comparison, it is the scheme of
Shi's derivative estimates in Hamilton's presentation \cite{ham95}, compare also \cite[Theorem~3.4]{eh91} and the
proof of \cite[Theorem~1.1(1)]{chwu}). Since $s_0\in[\sigma_m,s_0]$ for every $m$, we obtain $|\nabla^mB|(\cdot,s_0)\leq C_m$ on $\Sigma$ for
every $m\geq0$ and every $s_0\in[-S,-S^{-1}]$. These bounds
hold for the metric $g_\kappa$; transported by $\delta_\lambda$ they are bounds for $F^\lambda_{s_0}$ in the metric of $\sigma_\lambda$,
which converges to the Heisenberg metric on compact sets (Lemma~\ref{lem-normal}). In particular they hold at the
points of $F^\lambda_{s_0}$ with $|w|,|z|\leq R$.
\end{proof}

\begin{proof}[Proof of Theorem~\ref{thm-typeI} and Corollary~\ref{cor-planes} for arbitrary $M$]
The proof in Section~\ref{sec-typeI} uses regularity only through Lemma~\ref{lem-kernel}, Lemma~\ref{lem-compact} and, in the step
on self-similarity, Proposition~\ref{prop:mono}. Replace $\Theta$ by $\Theta_{\rm loc}$, $\nu$ by $\nu^{\rm box}$, and
Lemma~\ref{lem-compact} by Lemma~\ref{lem-compact-general}. The steps then go through as follows.
\begin{itemize}
\item \emph{Non-emptiness} uses only the coordinates of Lemma~\ref{lem-normal}.
\item \emph{Compactness.} Lemma~\ref{lem-immersion-compact} applies unchanged. Its proof uses the bounds of
Lemma~\ref{lem-compact} only on compact sets, together with the area bound over $\{|w|\leq R\}$, and Lemma~\ref{lem-compact-general}
provides these in the rescaled inner boxes $\{|w|<\lambda r,\ |z|<\lambda^2\delta\}$, which exhaust $\R^{2n+1}$.
\item \emph{Self-similarity.} For $t_j=T+s_j\lambda^{-2}$, Proposition~\ref{prop:mono-local} bounds the localised defect over
$[t_1,t_2]$ by
\[
\Theta_{\rm loc}(t_1)-\Theta_{\rm loc}(t_2)+C_0\int_{t_1}^{t_2}\Theta_{\rm loc}+\int_{t_1}^{t_2}E ,
\]
which tends to $0$, because $\lim\Theta_{\rm loc}$ exists and $E$ is integrable. On every compact set of the rescaled
coordinates the cut-offs equal $1$ for $\lambda$ large, and Lemma~\ref{lem-kernel}(iii) holds with $\iota_W$ in place of
$\iota$, its proof being local.
\item \emph{Number of components.} The argument uses $\nu^{\rm box}$. Its rescalings cover heights up to $\lambda^2\delta$, so the
limit satisfies $\mu_{\infty,s}(\{|w|<R,\ |z|<L\})\leq C_1R^n$ for every $L>0$; letting $L\to\infty$ after the limit gives
$\nu_{\infty,s}(B_R)\leq C_1R^n$ over all heights, as before.
\item \emph{Vertical centre, including the identification of the limit angle, and the planes of
Corollary~\ref{cor-planes}:} unchanged, since these steps use only Lemma~\ref{lem-immersion-compact},
Lemma~\ref{lem-ss-flow} and the angle bound $|\theta|\leq C_\theta$.
\end{itemize}
\end{proof}

\begin{remark}\label{rem-general-scope}
Theorem~\ref{thm-typeI} and Corollary~\ref{cor-planes} thus hold in every compact $\eta$-Einstein Sasakian manifold. Corollary~\ref{cor-typeI-intro} and Lemma~\ref{lem-eps} are stated for regular manifolds. They use the
transverse density $\Theta$, which counts all sheets over a whole Reeb circle and is an integer at Type I singular
points. In a flow box, the corresponding local density weights the sheets above the Reeb arc through $p$ by the
cut-off $\zeta$ and need not be an integer. A version with a sharp cut-off requires the height of the box to avoid the
closed set of heights at which $F(\cdot,T)$ meets that arc, and we do not pursue it.
\end{remark}

\section{Curvature estimates and exponential convergence}\label{sec-conv}

\begin{lemma}\label{lem-B-evol}
Along \eqref{c} or \eqref{c-norm} in $M_k$ there are constants $C_n,K_1,K_2$ depending only on $k$ such that
\[
\Big(\frac{\p}{\p t}-\Delta\Big)|B|^2\leq-2|\nabla B|^2+C_n|B|^4+K_1|B|^2+K_2 .
\]
\end{lemma}

\begin{proof}
This follows from the evolution equation of $|B|^2$ along \eqref{c} in an $\eta$-Einstein Sasakian manifold
\cite[Corollary~2.10(2)]{chwu}: its reaction terms are quartic in $h$, the term $6|H|^2$ is at most $6n|B|^2$, and
the remaining terms are bounded by the curvature of $M_k$ and its covariant derivative; the normalised flow
differs by a rotation along the Reeb flow, which is an isometry.
\end{proof}

\begin{proposition}\label{prop:A-bound}
Let $F:\Sigma\times[0,\infty)\to M_k$ be a solution of \eqref{c} or \eqref{c-norm} starting from an admissible Legendrian
graph. Then $|B|$ is uniformly bounded on $\Sigma\times[0,\infty)$.
\end{proposition}

\begin{proof}
Let $\delta_*,\gamma_*$ be as in Corollary~\ref{cor-unified}, let $\epsilon=\min\{\frac12,\frac{\delta_*}{2(C_n+\delta_*)}\}$, and
let $T_\epsilon$ be such that $1-\Omega\leq\frac\epsilon2$ for $t\geq T_\epsilon$. Then $\chi_\epsilon=\Omega-1+\epsilon\in[\frac\epsilon2,\epsilon]$
and $(\frac{\p}{\p t}-\Delta)\chi_\epsilon\geq\delta_*|B|^2\Omega$ for $t\geq T_\epsilon$. For a quotient,
$(\frac{\p}{\p t}-\Delta)\frac uv=\frac1v(\frac{\p}{\p t}-\Delta)u-\frac{u}{v^2}(\frac{\p}{\p t}-\Delta)v+\frac2v\la\nabla v,\nabla\frac uv\ra$;
with $u=|B|^2$, $v=\chi_\epsilon$, $\zeta=u/v$ and Lemma~\ref{lem-B-evol},
\[
\begin{aligned}
\Big(\frac{\p}{\p t}-\Delta\Big)\zeta&\leq\frac2{\chi_\epsilon}\la\nabla\chi_\epsilon,\nabla\zeta\ra+\big(C_n\chi_\epsilon-\delta_*\Omega\big)\zeta^2
+K_1\zeta+\frac{K_2}{\chi_\epsilon}\\
&\leq\frac2{\chi_\epsilon}\la\nabla\chi_\epsilon,\nabla\zeta\ra-\frac{\delta_*}2\zeta^2+K_1\zeta+\frac{2K_2}{\epsilon},
\end{aligned}
\]
because $C_n\chi_\epsilon-\delta_*\Omega\leq C_n\epsilon-\delta_*(1-\epsilon)\leq-\frac{\delta_*}2$. By the maximum principle $\zeta$ is bounded on
$[T_\epsilon,\infty)$ by $\max\{\max_{\Sigma_{T_\epsilon}}\zeta,\zeta_*\}$, where $\zeta_*$ is the positive root of
$\frac{\delta_*}2\zeta^2=K_1\zeta+2K_2/\epsilon$; hence $|B|^2\leq\epsilon\zeta$ is bounded there, and on $[0,T_\epsilon]$ by
smoothness. This is the argument of \cite[Section~3.4]{mw}, compare \cite[Proposition~7.2]{w4}.
\end{proof}

\begin{theorem}\label{Thm}
Let $\Sigma_0\subset M_k$ be an admissible Legendrian graph. The normalised flow \eqref{c-norm} starting from $\Sigma_0$
exists for all time and converges to a totally geodesic Legendrian graph $\Sigma_{\varphi_\infty}$ with $\varphi_\infty\in U(k+1)$,
exponentially in every $C^m$ norm: $\|F_t-F_\infty\|_{C^m}\leq C_me^{-\gamma_mt}$ for suitable parametrisations, with
$\gamma_m>0$ depending on $m$.
\end{theorem}

\begin{proof}
\emph{Uniform bounds.} Long-time existence is Theorem~\ref{thm-flat}. By Proposition~\ref{prop:A-bound} $|B|$ is bounded,
so all derivatives $|\nabla^mB|$ are bounded on $[1,\infty)$ \cite[Theorem~1.1(1)]{chwu}, and by
Corollary~\ref{cor-unified} the submanifolds are graphs of maps whose first derivatives are uniformly bounded; in local
coordinates the derivatives of order $m+2$ of these maps are bounded in terms of $\nabla^mB$ and the derivatives of lower
order, so by induction on $m$ all their derivatives are uniformly bounded.

\emph{Exponential decay of the curvature.} By Corollary~\ref{cor-unified} and $\frac{d}{dt}d\mu=-|H|^2d\mu$, since $0\leq1-\Omega\leq1$,
\[
\frac{d}{dt}\int_{\Sigma_t}(1-\Omega)\,d\mu\leq-\int_{\Sigma_t}\Big(\frac{\p}{\p t}-\Delta\Big)\Omega\,d\mu
\leq-\delta_*c_*\int_{\Sigma_t}|B|^2d\mu .
\]
As the volume is non-increasing and $1-\Omega\leq Ce^{-\gamma_*t}$,
$\int_t^\infty\!\int_{\Sigma_s}|B|^2d\mu\,ds\leq A\,e^{-\gamma_*t}$ with $A=C|\Sigma_0|/(\delta_*c_*)$. The function
$\beta(t)=\int_{\Sigma_t}|B|^2d\mu$ satisfies $|\beta'|\leq L$ by the uniform bounds. For $h>0$,
$h\,\beta(t)\leq\int_t^{t+h}\beta+\frac12Lh^2$, and the choice $h=\sqrt{2A/L}\,e^{-\gamma_*t/2}$ gives
$\beta(t)\leq\sqrt{2AL}\,e^{-\gamma_*t/2}$. Since $|\nabla|B||\leq|\nabla B|\leq C$ and the submanifolds have uniformly
bounded geometry, a point where $|B|=b$ is the centre of an intrinsic ball of radius $b/(2C)$ on which $|B|\geq b/2$
and whose volume is at least $c\,b^{2k}$; hence $b^{2k+2}\leq C'\beta(t)$ and
$\sup_{\Sigma_t}|B|\leq C''e^{-\gamma_*t/(4k+4)}$.

\emph{Convergence and the limit.} Then $|H|$ and $|\p_tg|\leq2|H||B|$ decay exponentially, so the induced metrics are uniformly
equivalent and the diameters bounded, and $\mathrm{osc}_{\Sigma_t}\theta\leq\sup|\nabla\theta|\,\mathrm{diam}(\Sigma_t)$ decays
exponentially, since $|\nabla\theta|=|H|$. The velocity of \eqref{c-norm} therefore decays exponentially, so that $F_t$
converges uniformly and exponentially as $t\to\infty$; by interpolation between this $C^0$ estimate and the uniform
bounds on all derivatives, the convergence is exponential in every $C^m$ norm, with a rate depending on $m$. The limit is totally geodesic and has $\Omega\equiv1$ by Corollary~\ref{cor-unified}, so it is the
Legendrian graph of a strict contactomorphism $\varphi_\infty$ all of whose singular values equal $1$; thus
$\tilde\varphi_\infty$ is a symplectic isometry of $\mathbb{CP}^k$, hence holomorphic, and lies in $PU(k+1)$. The strict
contactomorphisms covering $\tilde\varphi_\infty$ differ by the Hopf action, and one of them lies in $U(k+1)$; hence
$\varphi_\infty\in U(k+1)$.
\end{proof}

\begin{remark}
The unnormalised flow \eqref{c} need not converge. By \eqref{d} with $K+2=2(k+1)>0$, the mean value of $\theta$ is
not preserved; a minimal Legendrian submanifold with constant angle $\theta\neq0$ is not stationary for \eqref{c} but
moves along the Reeb flow with speed $2\theta$, and $\theta$ grows. This is why Theorem~\ref{thm-smale} is stated for
\eqref{c-norm}.
\end{remark}

\begin{proof}[Proof of Theorem~\ref{thm-smale} and Corollary~\ref{cor-smale}]
By Proposition~\ref{pro-graph}, $\Sigma_{\varphi_0}$ is a Legendrian graph, and it is admissible: for $k=1$ by
definition, and for $k\geq2$ because $\varphi_0$ is $\Lambda$-pinched with $\Lambda\leq\Lambda_*(k)$ (Proposition~\ref{prop:pinch}). By
Corollary~\ref{cor-unified} and Theorem~\ref{Thm} the normalised flow consists of Legendrian graphs $\Sigma_{\varphi_t}$
and converges to $\Sigma_{\varphi_\infty}$ with $\varphi_\infty\in U(k+1)$. The family $\varphi_t$ is the required isotopy. With $t=\tau/(1-\tau)$, $\tau\in[0,1)$, and $\varphi_{t(1)}=\varphi_\infty$, the
map $\tau\mapsto\varphi_{t(\tau)}$ is continuous on $[0,1]$ in the $C^\infty$ topology, by Proposition~\ref{pro-graph} and the
exponential convergence in every $C^m$ norm. It is also smooth in $\tau$ up to $\tau=1$: by \eqref{c-norm} every time derivative
of the flow is a combination of spatial derivatives of the velocity, which decay exponentially in every $C^m$ norm,
whereas the derivatives of $t(\tau)$ grow only polynomially in $t$.
\end{proof}

\begin{remark}[Relation to \cite{w1,mw}]\label{rem-lift}
By Lemma~\ref{lem-proj} the maps $\tilde\varphi_t$ form the isotopy of \cite[Theorem~A and Theorem~3.1]{w1} and
\cite[Theorem~1]{mw}. Since $H^1(\mathbb{CP}^k;\R)=0$ this isotopy is Hamiltonian, and its lift through the connection
$\eta_0$, with initial value $\varphi_0$, is a strict contact isotopy $\varphi^H_t$ covering $\tilde\varphi_t$. As $\varphi_t$ covers
$\tilde\varphi_t$ as well, $\varphi_t=R_{\alpha(t)}\circ\varphi^H_t$ for a function $\alpha$, where $R_\alpha$ is the rotation by $\alpha$ along the
Reeb flow; the Legendrian flow thus selects a particular phase $\alpha(t)$. In particular, the existence of an isotopy
from $\varphi_0$ to $U(k+1)$ through strict contactomorphisms also follows from \cite{w1,mw} by lifting; for $k=1$ a lift
of Wang's flow for the whole group, horizontal for an $L^2$ metric, is constructed in \cite[Theorem~A]{dglmwy}.
\end{remark}

\section*{Acknowledgements}
The first author is partially supported by the Startup Foundation for Advanced Talents of the Shanghai Institute for Mathematics and Interdisciplinary Sciences (Grant No.~\mbox{2302-SRFP-2024-0049}); the second author by the National Science and Technology Council, Taiwan (Grant No.~\mbox{NSTC 114-2115-M-153-001}); the third author by the National Natural Science Foundation of China (Grant No.~12601100).
Part of this work was carried out while the first-named
author was visiting the Department of Mathematics of National Taiwan Normal University, supported by the iCAG
program of the Higher Education Sprout Project of National Taiwan Normal University and the Ministry of Education
(MOE) of Taiwan. He thanks the department for its warm hospitality.

\section*{AI Declaration}
During the preparation of this paper the authors used Claude Opus 5.5 (Anthropic) and
GPT-6 Astra (OpenAI), the latter accessed through Codex. These tools were used to check the
proofs and computations line by line, to check statement numbers and citations against the
cited sources, and to improve the language and the LaTeX layout. They were not used to
formulate the results. Every argument was checked independently by the authors, who take
full responsibility for the content of the paper.

\bibliographystyle{amsalpha}
\bibliography{ContactBlowup_LMCF_SmaleType}

\end{document}